\documentclass[11pt]{amsart}

\usepackage[T1]{fontenc}
\usepackage[utf8]{inputenc}

\usepackage{amsmath,amssymb,amsfonts,amsthm}
\usepackage{mathrsfs}
\usepackage{hyperref}
\usepackage{graphicx}
\usepackage{float}
\usepackage{placeins}
\usepackage{tikz}
\usetikzlibrary{decorations.pathmorphing}

\numberwithin{equation}{section}

\newcommand{\M}{\mathcal{M}}

\newcommand{\Hh}{\mathbb{H}}
\newcommand{\Pp}{\mathbb{P}}
\DeclareMathOperator{\End}{End}
\DeclareMathOperator{\SL}{SL}
\DeclareMathOperator{\GL}{GL}
\DeclareMathOperator{\Wr}{Wr}

\DeclareMathOperator{\tr}{tr}

\theoremstyle{plain}
\newtheorem{theorem}{Theorem}[section]
\newtheorem{lemma}[theorem]{Lemma}
\newtheorem{proposition}[theorem]{Proposition}
\newtheorem{corollary}[theorem]{Corollary}

\theoremstyle{definition}
\newtheorem{definition}[theorem]{Definition}
\newtheorem{example}[theorem]{Example}

\theoremstyle{remark}
\newtheorem{remark}[theorem]{Remark}

\title[A Non-Abelian Approach to Riemann Surfaces]{A Non-Abelian Approach to Riemann Surfaces}
\author{Mehrzad Ajoodanian}

\date{}

\newcommand{\Cc}{\mathbb{C}}

\newcommand{\Zz}{\mathbb{Z}}
\begin{document}

\begin{abstract}
We develop a non-abelian, gauge-theoretic framework for the Schwarzian derivative and for second-order differential equations on Riemann surfaces.
As applications, we extend Dedekind's Schwarzian approach to elliptic periods to generic one-parameter families of curves of genus $g$ by replacing the non-canonical scalar Picard--Fuchs equation of order $2g$ with a canonical second-order equation with $g\times g$ matrix coefficients on the Hodge bundle. In higher dimensions, we discuss periods of a one-parameter family of cubic threefolds via the intermediate Jacobian.
Finally, we discuss mass--spring systems in mechanics as a natural testing ground for the non-abelian Schwarzian viewpoint.
\end{abstract}

\maketitle

% Unnumbered acknowledgment footnote on the first page.
\begingroup
\renewcommand\thefootnote{}
\footnotetext{I would like to thank Amir Jafari.}
\addtocounter{footnote}{-1}
\endgroup

\begin{flushright}
\emph{``Everyone knows what a curve is until he has studied enough mathematics.''}\\
--- Felix Klein
\end{flushright}

\tableofcontents

\section{Introduction}

In a letter to Borchardt published in 1877, Dedekind discovered a striking relation between the Schwarzian derivative and ratios of elliptic integrals \cite[\S\S~5--6]{L4}. One considers the periods of an elliptic curve as functions of the normalized Hauptmodul $j$ on the modular curve $\Hh/\SL_2(\mathbb{Z})$. These periods satisfy a second-order Picard--Fuchs equation (see, e.g., \cite{L9,L10})
\[
\frac{d^2y}{dj^2} + \frac{1}{j}\frac{dy}{dj} + \frac{31j-4}{144j^2(1-j)^2}y = 0.
\]
If $\omega_1(j),\omega_2(j)$ are two independent solutions, their ratio
\[
\tau(j)=\frac{\omega_2(j)}{\omega_1(j)}
\]
is a local uniformizing parameter on $\Hh$. The ordered pair $(\omega_1,\omega_2)$ is determined only up to the natural action of $\GL_2(\mathbb{C})$ on the two-dimensional space of solutions; hence $\tau$ is determined only up to a M\"obius transformation. Recall that for a locally univalent meromorphic function $f$ one defines the Schwarzian derivative by
\[
S(f):=\left(\frac{f''}{f'}\right)'-\frac12\left(\frac{f''}{f'}\right)^2.
\]
Following Klein's account in \cite[Chapter~III, \S6]{L32}, this invariant goes back to Lagrange (see \cite{L33}).
Since the Schwarzian derivative is invariant under M\"obius transformations, the quantity $S(\tau)$ is well defined, independent of the chosen basis. Dedekind's computation asserts that $S(\tau)$ is in fact the rational function
\[
S(\tau)(j)=\frac{3}{8(1-j)^2}+\frac{4}{9j^2}+\frac{23}{72j(1-j)}.
\]

\medskip
\noindent
\textbf{From Schwarzian to curvature.}
A guiding principle of this paper is that Dedekind's identity is best understood as a curvature statement for a connection naturally attached to the differential equation. Concretely, the pre-Schwarzian  $ A(f)=f''/f'$ behaves like a Maurer--Cartan (quantum) connection, and its curvature $F_A=A'-A^2$ is the Schwarzian derivative. On a general Riemann surface the curvature fails to be a genuine quadratic differential: it transforms with a controlled \emph{projective anomaly} given by the Schwarzian of the coordinate changes. This anomaly is precisely the transformation law of projective connections \cite{L13}.

\medskip
\noindent
The main lemma and theorem show that if $g$ is an invertible matrix-valued solution of $g''=2Ag'+qg$, then under the gauge action, the connection $g^{-1}\bullet A$ has curvature
\[
F_{g^{-1}\bullet A}=g^{-1}(F_A-q)g.
\]
As a corollary, the conjugacy invariants of $F_A-q$ (traces of powers, coefficients of characteristic polynomials) define invariants of the differential equation; we call them \emph{characteristic invariants} by analogy with Chern's characteristic forms \cite{L16}. In the scalar case this reduces to Dedekind's classical mechanism: the Schwarzian of the ratio of periods is determined by the coefficients of the Picard--Fuchs equation.

\medskip
\noindent
Some of the ideas developed here continue in a higher-order direction in recent work of Jafari \cite{L36}, where the reparametrization formalism is extended to higher Schwarzian cocycles and noncommutative Wilczynski-type invariants.

\medskip
\noindent
\textbf{Applications:}

\medskip
\noindent
We discuss three applications of the non-abelian Schwarzian formalism.
The same formalism produces canonical second-order systems in three different settings: periods of curves, periods of higher-dimensional varieties, and mechanical systems.
The higher-genus case is the main geometric application: it replaces a non-canonical scalar Picard--Fuchs equation of order $2g$ by a canonical second-order equation with $g\times g$ matrix coefficients on the Hodge bundle.

\begin{enumerate}
\item[(i)] \emph{Periods of Riemann surfaces of higher genus} (Sections~\ref{sec:higher-genus-periods} and~\ref{sec:genus-two-family}).
Periods in a holomorphic family satisfy Picard--Fuchs equations arising from the Gauss--Manin connection (see \cite{L9,L10,L15}).
For a one-parameter family of genus $g$ curves, the classical reduction typically yields a non-canonical scalar equation of order $2g$ \cite{L10}; for a \emph{generic} family (Definition~\ref{def:generic-family}) we instead construct canonically a second-order equation with $g\times g$ matrix coefficients on the Hodge bundle, whose solutions are period vectors.
This canonical system is the natural higher-genus analogue of Dedekind's genus-one picture, and its associated matrix Schwarzian produces the corresponding characteristic invariants.

\item[(ii)] \emph{Periods of higher-dimensional varieties: cubic threefolds} (Section~\ref{sec:cubic-threefolds}).
For a one-parameter family of smooth cubic threefolds in $\mathbb{P}^4$, the Gauss--Manin connection gives a first-order system on $H^3$ of rank $10$.
Our formalism again produces a canonical second-order system, now on the rank-$5$ Hodge bundle $E=H^{2,1}$, showing that the same Schwarzian viewpoint extends beyond curves to a higher-dimensional variation of Hodge structure.

\item[(iii)] \emph{Mass--spring systems} (Section~\ref{sec:mass-spring-system}).
Mass--spring systems naturally give rise to second-order ordinary differential equations with matrix-valued coefficients.
They provide a concrete mechanical model for the same formalism: the coefficients are most naturally viewed as geometric data on a time curve, transforming coherently under changes of local time coordinate.

\end{enumerate}

\medskip
\noindent
\textbf{Organization of the paper.}
The paper is divided into five parts.

\smallskip
\noindent
\emph{Part~I} (Sections~2--3) recalls the classical Schwarzian invariant of a second-order scalar ODE and interprets it as curvature.

\smallskip
\noindent
\emph{Part~II} (Sections~4--8) develops a (quantum) theory of connections and differentials on Riemann surfaces.

\smallskip
\noindent
\emph{Part~III} (Section~11) discusses modular forms and modular connections.

\smallskip
\noindent
\emph{Part~IV} (Sections~12--14) treats periods of Riemann surfaces and cubic threefolds from a Schwarzian perspective.

\smallskip
\noindent
\emph{Part~V} (Section~15) discusses mass--spring systems in mechanics and their relation to the Schwarzian.

\part{Classical Schwarzian Derivative}

\section{Second order scalar ODE}\label{sec:classical-ode-schwarzian}

Let $U\subset \Cc$ be a domain with coordinate $z$, and let $p,q\in \M(U)$ be meromorphic functions.
We consider the second order linear equation
\begin{equation}\label{eq:classical-second-order-pq}
y''=2p\,y' + q\,y.
\end{equation}
If $y_1,y_2$ are two independent solutions, then the ratio $y_1/y_2$ is locally
univalent and is well defined only up to a M\"obius transformation (change of
basis in the two-dimensional solution space). The Schwarzian derivative,
\[
S(w):=\frac{w'''}{w'}-\frac{3}{2}\Bigl(\frac{w''}{w'}\Bigr)^2
=\Bigl(\frac{w''}{w'}\Bigr)'-\frac12\Bigl(\frac{w''}{w'}\Bigr)^2,
\]
is invariant under M\"obius transformations, hence $S(y_1/y_2)$ is canonically
attached to the equation.

The following classical computation is standard (see, for example,
\cite[\S\,IV.2]{L13}). For a classical discussion of the
projective viewpoint on second-order equations, see also Klein's
\emph{Lectures on the Icosahedron} \cite[Chapter~III, \S7]{L32}.

\begin{proposition}\label{prop:classical-schwarzian-ode}
Let $y_1,y_2$ be two linearly independent solutions of
\eqref{eq:classical-second-order-pq}. Then the Schwarzian derivative
$S(y_1/y_2)$ is independent of the chosen ordered basis of solutions and satisfies
\begin{equation}\label{eq:classical-schwarzian-formula}
S\Bigl(\frac{y_1}{y_2}\Bigr)=2\bigl(p'-p^2-q\bigr).
\end{equation}
\end{proposition}

\noindent\emph{Comment.}
The right-hand side $2(p'-p^2-q)$ is invariant under gauge changes of the dependent
variable $y\mapsto f\,y$ (with $f$ nowhere vanishing), and under a change of
coordinate it transforms with the usual Schwarzian anomaly. This is precisely the
``projective curvature'' viewpoint that motivates the non-abelian generalizations
developed later.

The identity \eqref{eq:classical-schwarzian-formula} implies that
$2(p'-p^2-q)$ is invariant under the two basic operations on
\eqref{eq:classical-second-order-pq}:

\begin{itemize}
\item[(i)] \emph{Gauge (change of dependent variable).}
If $y=f\,\widetilde y$ with $f$ nowhere vanishing, then $\widetilde y$ satisfies
$\widetilde y''=2\widetilde p\,\widetilde y'+\widetilde q\,\widetilde y$ with
\[
\widetilde p=p-\frac{f'}{f},
\qquad
\widetilde q=q+2p\frac{f'}{f}-\frac{f''}{f},
\]
and one checks directly that
$\widetilde p'-\widetilde p^{\,2}-\widetilde q=p'-p^2-q$.

\item[(ii)] \emph{Change of coordinate.}
If $z=\lambda(\zeta)$ is a biholomorphism and $\widetilde y(\zeta)=y(\lambda(\zeta))$,
then $\widetilde y$ satisfies $\widetilde y''=2\widetilde p\,\widetilde y'+\widetilde q\,\widetilde y$
with
\[
\widetilde p=(p\circ\lambda)\lambda'+\frac12\frac{\lambda''}{\lambda'},
\qquad
\widetilde q=(q\circ\lambda)(\lambda')^2.
\]
In particular,
\[
2(\widetilde p'-\widetilde p^{\,2}-\widetilde q)
=(2(p'-p^2-q)\circ\lambda)(\lambda')^2+S(\lambda).
\]
\end{itemize}
Thus $2(p'-p^2-q)$ transforms as a projective connection.

\section{Projective curvature and the Schwarzian anomaly}\label{sec:projective-curvature}

Equation \eqref{eq:classical-second-order-pq} suggests that the data of $(p,q)$ should be
organized so that the projective curvature \eqref{eq:classical-schwarzian-formula} becomes intrinsic.

\subsection*{Connections of eccentricity 1/2.}
The transformation law in (ii) shows that $p$ behaves like a connection coefficient: under
$z=\lambda(\zeta)$ one has
\[
\widetilde p=(p\circ\lambda)\lambda'+\frac12\frac{\lambda''}{\lambda'}.
\]
This is the one-dimensional avatar of the usual gauge transformation law for a connection,
except that the change of coordinate contributes an inhomogeneous term. The corresponding
(curvature-type) combination is $p'-p^2$.

\subsection*{Projective anomaly.}
The expression
\[
r:=2(p'-p^2-q)
\]
is the classical projective curvature attached to \eqref{eq:classical-second-order-pq}; it
equals $S(y_1/y_2)$ by Proposition~\ref{prop:classical-schwarzian-ode}.
Under coordinate change $z=\lambda(\zeta)$ it satisfies
\[
\widetilde r=(r\circ\lambda)(\lambda')^2+S(\lambda),
\]
so $r$ is not a quadratic differential in general, but it fails to be one by a \emph{scalar}
term given by the Schwarzian derivative of the transition map. In particular, if all
coordinate changes are M\"obius transformations (so $S(\lambda)=0$), then $r$ is a genuine
quadratic differential. This is the mechanism behind the special features of the modular
case.

\subsection*{Schwarzian as curvature of a connection.}
Let $f$ be a locally univalent meromorphic function on $U\subset\Cc$. On the open set where
$f'\neq 0$ the logarithmic derivative
\[
A(f)=\frac{f''}{f'}
\]
behaves like a connection: for any biholomorphism $\lambda$ one has the coordinate-change rule
\[
A(f\circ\lambda)
=\lambda'\bigl(A(f)\circ\lambda\bigr)+\frac{\lambda''}{\lambda'}.
\]
The curvature-type expression
\[
F_A(f)=\left(\frac{f''}{f'}\right)'-\frac12\left(\frac{f''}{f'}\right)^2
\]
is exactly the Schwarzian derivative $S(f)$. The projective anomaly is the Schwarzian chain
rule
\[
S(f\circ\lambda)=(\lambda')^2\bigl(S(f)\circ\lambda\bigr)+S(\lambda),
\]
and $S$ is invariant under postcomposition of $f$ by a M\"obius transformation. This point
of view motivates the non-abelian constructions developed in the next part.

\part{Quantum vs classical}

\section{Classical differentials and connections}\label{sec:classical-differentials-connections}

Let $X$ be a Riemann surface. Choose a covering by coordinate charts
$\{U_i\}_i$ on $X$. By abuse of notation we identify each $U_i$ with an open subset
of $\mathbb{C}$. For each pair $i,j$ with $U_i\cap U_j\neq\varnothing$, the change of
coordinate induces a biholomorphism
\[
\lambda_{ij} : U_i\cap U_j \longrightarrow U_i\cap U_j,
\]
where the domain is viewed as an open subset of $\mathbb{C}$ via the coordinate on
$U_j$, and the codomain is viewed as an open subset of $\mathbb{C}$ via the coordinate
on $U_i$. The maps $\lambda_{ij}$ satisfy the usual cocycle conditions.

Recall that $K_X$ denotes the canonical line bundle of holomorphic $1$-forms
on $X$. Fix an integer $m$. A meromorphic connection on $K_X^m$ may be encoded
locally by meromorphic functions $a_i\in\mathcal{M}(U_i)$ satisfying on overlaps
\[
a_j = \lambda'_{ij}(a_i\circ \lambda_{ij}) - m\frac{\lambda''_{ij}}{\lambda'_{ij}}.
\]

A \emph{classical meromorphic $m$-differential} on $X$ is a meromorphic section of $K_X^m$.
Equivalently, it is given by meromorphic functions $\omega_i\in \M(U_i)$ satisfying on overlaps
\[
\omega_j = (\lambda'_{ij})^{m}\,(\omega_i\circ \lambda_{ij}).
\]
More generally, if $V$ is a finite-dimensional complex vector space, a $V$-valued classical
$m$-differential is a meromorphic section of $K_X^m\otimes V$, and in local coordinates it is
represented by meromorphic maps $\omega_i:U_i\to V$ with the same tensorial rule.

Similarly, a \emph{meromorphic connection} on $K_X^m$ may be encoded by meromorphic functions
$a_i\in \M(U_i)$ satisfying
\[
a_j = \lambda'_{ij}(a_i\circ \lambda_{ij}) - m\frac{\lambda''_{ij}}{\lambda'_{ij}}.
\]
If $e\in\Zz$, this is the same as a connection on $K_X^{-e}$ with $m=-e$; the parameter $e$ is
the \emph{eccentricity} used later in the quantum transformation law.

When we allow coefficients in $\End(V)$, the classical objects above correspond exactly to the
$\mathcal O$-linear quantum objects introduced in Section~\ref{sec:quantum-connections} below:

\begin{example}
A quantum connection (resp.\ a quantum $m$-differential) with values in $V$ is called classical with values in $\End(V)$ if it is $\mathcal{O}_X$-linear.
Equivalently, it is given by multiplication by a coefficient $a_i\in\mathcal{M}(U_i,\End(V))$ (resp.\ $\psi_i\in\mathcal{M}(U_i,\End(V))$):
\[
A_i(f)=a_if,\qquad \Psi_i(f)=\psi_if.
\]
The coefficients satisfy the corresponding change of coordinate rules. In particular, an $\mathcal{O}_X$-linear quantum connection of eccentricity $e$ corresponds to
a classical connection on $K_X^{-e}$ with coefficients in $\End(V)$, and an $\mathcal{O}_X$-linear quantum $m$-differential corresponds to a classical section of
$K_X^{m}\otimes \End(V)$.
\end{example}

This classical vs.\ quantum distinction is important throughout the paper: in many
geometric situations the most natural operations (Maurer--Cartan connections, Wronskians...) are not $\mathcal O$-linear.

\section{Quantum differentials and connections }\label{sec:quantum-connections}

We now pass from $\mathcal O$-linear tensors to operator-valued objects. The basic input is
that many natural constructions (Maurer--Cartan, Wronskians, exponentials, etc.) are only
defined on Zariski open subsets depending on the input function. The following definition
packages this ``generic'' behavior in a coordinate-free way.

Many of the operators that appear naturally in this paper (for instance
$f\mapsto f''/f'$ or $f\mapsto e^f$) do not produce a meromorphic function on the
\emph{entire} chart $U_i$ even when the input $f$ is meromorphic on $U_i$. Rather,
the output is meromorphic on a Zariski open subset obtained by deleting a
discrete set of points depending on $f$ (for $f''/f'$ one removes the zeros of
$f'$, while for $e^f$ one removes the poles of $f$).

To keep the formalism honest while retaining the coordinate-free cocycle laws,
we systematically work ``generically'' on charts.
If $U\subset X$ is open, we call $U^\circ\subset U$ \emph{Zariski open} if
$U\setminus U^\circ$ is a discrete subset. For a complex vector space $V$ we set
\[
\M_{\mathrm{gen}}(U,V)
:=\varinjlim_{U^\circ\subset U\ \text{Zariski open}}\M(U^\circ,V),
\]
the direct limit with respect to restriction. Thus an element of
$\M_{\mathrm{gen}}(U,V)$ is represented by a $V$-valued meromorphic function on
some Zariski open subset of $U$, and two representatives are identified if they
agree after restriction to a smaller Zariski open subset.

All identities below involving quantum connections and quantum differentials are
understood in this ``generic'' sense: the equality of two expressions means that
they agree on some Zariski open subset on which both sides are defined. In
degenerate cases (for example $f''/f'$ when $f$ is constant) the natural domain
of definition may be empty; such exceptional inputs occur in very special
situations and will be ignored.

\begin{definition}\label{def:quantum-connection}
Fix $e\in\mathbb{C}$ (the eccentricity) and a complex vector space $V$.
A quantum connection on $X$ of eccentricity $e$ with values in $V$ is a family
of local operators
\[
A_i : \M_{\mathrm{gen}}(U_i,V)\dashrightarrow \M_{\mathrm{gen}}(U_i,V),
\]
such that on every overlap $U_i\cap U_j$ one has, as operators on $\M_{\mathrm{gen}}(U_i\cap U_j,V)$,
\[
A_j = \lambda'_{ij}(A_i\circ \lambda_{ij}) + e\frac{\lambda''_{ij}}{\lambda'_{ij}}I.
\]
Here $I$ denotes the identity operator on $\M_{\mathrm{gen}}(U_i\cap U_j,V)$, scalar functions act
by pointwise multiplication, and the transported operator is
\[
(A_i\circ \lambda_{ij})(f) := \bigl(A_i(f\circ \lambda_{ij}^{-1})\bigr)\circ \lambda_{ij}.
\]
\end{definition}

\begin{remark}\label{rem:maurer-cartan-fpp}
In the scalar case, the Maurer--Cartan operator $A(f):=f''/f'$ (defined on the Zariski open subset where $f'\neq 0$)
is a basic example of a quantum connection of eccentricity $e=1$. Indeed, for a change of coordinate $\lambda$
one has $(f\circ\lambda)'=(f'\circ\lambda)\lambda'$ and
\[
\frac{(f\circ\lambda)''}{(f\circ\lambda)'}=\lambda'\left(\frac{f''}{f'}\circ\lambda\right)+\frac{\lambda''}{\lambda'},
\]
i.e.\ $A(f\circ\lambda)=\lambda'\,(A(f)\circ\lambda)+\lambda''/\lambda'$.
\end{remark}

\begin{definition}\label{def:quantum-m-differential}
Fix an integer $m$ and a complex vector space $V$. A quantum $m$-differential on
$X$ with values in $V$ is a family of local operators
\[
\Psi_i : \M_{\mathrm{gen}}(U_i,V)\dashrightarrow \M_{\mathrm{gen}}(U_i,V),
\]
such that on every overlap $U_i\cap U_j$ one has, as operators on $\M_{\mathrm{gen}}(U_i\cap U_j,V)$,
\[
\Psi_j = (\lambda'_{ij})^{m}(\Psi_i\circ \lambda_{ij}),
\]
where
\[
(\Psi_i\circ \lambda_{ij})(f) := \bigl(\Psi_i(f\circ \lambda_{ij}^{-1})\bigr)\circ \lambda_{ij}.
\]
In the definition of a quantum $m$-differential one may more generally allow
operators
\[
\Psi_i : \M_{\mathrm{gen}}(U_i,V)\dashrightarrow \M_{\mathrm{gen}}(U_i,W),
\]
for (possibly different) complex vector spaces $V$ and $W$; the same transformation rule applies.
\end{definition}

\subsection*{Examples.}

\begin{example}
Let $\omega$ be an ordinary meromorphic $m$-differential on $X$. In local coordinates it is given by meromorphic functions $\omega_i$ satisfying $\omega_j =
(\lambda'_{ij})^{m}(\omega_i\circ \lambda_{ij})$. Then multiplication by $\omega$,
\[
\Psi_i(f) := \omega_i f,
\]
defines a quantum $m$-differential (with values in any $V$).
\end{example}

\begin{example}
Assume $m\ge 0$. Define locally $\Psi_i(f):=(f')^{m}$. Then $\Psi$ is a quantum
$m$-differential since $(f\circ \lambda)' = (f'\circ \lambda)\lambda'$ implies
$((f\circ \lambda)')^{m} = (\lambda')^{m}((f')^{m}\circ \lambda)$.
\end{example}

\begin{example}[Wronskian for vector-valued functions]
Let $V$ be a complex vector space of dimension $n$. For $f\in\mathcal{M}(U_i,V)$ define
\[
W_i(f) := f\wedge f'\wedge f''\wedge \cdots \wedge f^{(n-1)}\in \bigwedge^{n}V.
\]
Then $W=\{W_i\}$ is a quantum $\frac{n(n-1)}{2}$-differential with values in
$\bigwedge^{n}V$.
\end{example}

\begin{example}[Wronskian of a quantum differential]
Let $r\in\mathbb{Z}$ and let $V$ be a complex vector space of dimension $n$.
Suppose $\Psi$ is a quantum $r$-differential with values in $V$. For $k\ge 0$ define
$\Psi^{(k)}$ locally by $\Psi^{(k)}_i(f) := (\Psi_i(f))^{(k)}$. Note that $\Psi^{(1)}$ need not be a quantum differential. Nevertheless the Wronskian
\[
\Wr(\Psi)_i(f) := \Psi_i(f)\wedge \Psi^{(1)}_i(f)\wedge \cdots \wedge \Psi^{(n-1)}_i(f)\in \bigwedge^{n}V
\]
defines a quantum $\bigl(nr+\frac{n(n-1)}{2}\bigr)$-differential with values in
$\bigwedge^{n}V$.
\end{example}

\begin{lemma}[Left Maurer--Cartan connection]
Let $V$ be a finite-dimensional complex vector space and let $\Psi=\{\Psi_i\}$ be a quantum $m$-differential on $X$ with values in $\End(V)$, where $m\neq 0$.
For each $i$ define
\[
A_i(f) := \frac{1}{m}\Psi_i(f)^{-1}\Psi'_i(f)
\]
on the open subset of $U_i$ where $\Psi_i(f)$ is pointwise invertible. Then
$A=\{A_i\}$ is a quantum connection of eccentricity $1$ with values in $\End(V)$.
\end{lemma}

\begin{proof}
Differentiate the relation $\Psi_j = (\lambda')^{m}(\Psi_i\circ \lambda)$ and multiply on the
left by $\Psi_j^{-1}$. A direct computation shows that the logarithmic derivative
$\frac{1}{m}\Psi^{-1}\Psi'$ transforms as $A_j = \lambda'(A_i\circ \lambda) + \frac{\lambda''}{\lambda'}I$. \qedhere
\end{proof}

\begin{example}[A scalar connection from a Wronskian]
Let $m\in\mathbb{Z}$ and let $V$ be a complex vector space of dimension $n$. If $\Psi$ is a quantum $m$-differential with values in $V$, then the previous example shows that $\Wr(\Psi)$ is a quantum
\[
N := mn + \frac{n(n-1)}{2} = \frac{n(n+2m-1)}{2}
\]
-differential with values in the one-dimensional space $\bigwedge^{n}V$. On the open subset where $\Wr(\Psi)(f)\neq 0$ one may form the scalar quotient
\[
\frac{\bigl(\Wr(\Psi)_i(f)\bigr)'}{\Wr(\Psi)_i(f)}.
\]
Defining
\[
A_i(f) := \frac{1}{N}\frac{\bigl(\Wr(\Psi)_i(f)\bigr)'}{\Wr(\Psi)_i(f)}
= \frac{2}{n(n+2m-1)}\frac{\bigl(\Wr(\Psi)_i(f)\bigr)'}{\Wr(\Psi)_i(f)},
\]
one obtains a scalar quantum connection of eccentricity $1$.
\end{example}

\begin{example}[Classical connections and $m$-differentials with values in $\End(V)$]
A quantum connection (resp.\ a quantum $m$-differential) with values in $V$ is called classical with values in $\End(V)$ if it is $\mathcal{O}_X$-linear.
Equivalently, it is given by multiplication by a coefficient $a_i\in\mathcal{M}(U_i,\End(V))$ (resp.\ $\psi_i\in\mathcal{M}(U_i,\End(V))$):
\[
A_i(f)=a_if,\qquad \Psi_i(f)=\psi_if.
\]
The coefficients satisfy the corresponding change of coordinate rules. In particular, an $\mathcal{O}_X$-linear quantum connection of eccentricity $e$ corresponds to
a classical connection on $K_X^{-e}$ with coefficients in $\End(V)$, and an $\mathcal{O}_X$-linear quantum $m$-differential corresponds to a classical section of
$K_X^{m}\otimes \End(V)$.
\end{example}

\section{Curvature}\label{sec:curvature-gauge}

The classical identity $S(f)=(f''/f')'-\tfrac12(f''/f')^2$ suggests that the correct
``curvature'' of an operator-valued connection should involve both a derivative term and
a quadratic term. In the quantum setting we also have to keep track of the coordinate-change
anomaly.

\begin{definition}\label{def:quantum-covariant-derivative}
Let $A$ be a quantum connection on $X$ with values in $\End(V)$ and eccentricity $e\neq 0$.
For $m\in\mathbb{Z}$ and a quantum $m$-differential $\Psi$ with values in $V$, define the quantum covariant derivative $\nabla_A\Psi$ locally by
\[
(\nabla_A\Psi)_i(f) := \Psi_i(f)' - \frac{m}{e}A_i(f)\Psi_i(f).
\]
\end{definition}

\begin{proposition}
If $A$ has eccentricity $e\neq 0$ and $\Psi$ is a quantum $m$-differential, then $\nabla_A\Psi$ is a quantum $(m+1)$-differential.
\end{proposition}

\begin{proof}
On an overlap with change of coordinate $\lambda$, differentiate the identity
$\Psi_j = (\lambda')^{m}(\Psi_i\circ \lambda)$ and use the transformation rule
$A_j = \lambda'(A_i\circ \lambda) + e\frac{\lambda''}{\lambda'}I$. The coefficient $\frac{m}{e}$ is chosen precisely so that the $\frac{\lambda''}{\lambda'}$-terms cancel, yielding
\[
(\nabla_A\Psi)_j = (\lambda')^{m+1}\bigl((\nabla_A\Psi)_i\circ \lambda\bigr).
\]
\qedhere
\end{proof}

\begin{remark}[Classical vs.\ quantum]
Even if $A$ and $\Psi$ are classical in the sense of the previous example, the operator $\nabla_A\Psi$ need not be classical: the derivative term $\Psi_i(f)'$
typically produces an $f'$-term. In the classical setting one instead defines the coefficient-level covariant derivative by
\[
(\nabla_a\psi)_i := \psi'_i - \frac{m}{e}a_i\psi_i,
\]
which is again classical. Thus the classical and quantum covariant derivatives are distinct constructions.
\end{remark}

We briefly recall the classical Cartan structural equation for a connection with values in a Lie algebra. Let $G$ be a Lie group with Lie algebra $\mathfrak{g}$.
A $\mathfrak{g}$-valued connection $1$-form on a manifold $M$ is a $\mathfrak{g}$-valued differential $1$-form $A\in\Omega^{1}(M,\mathfrak{g})$.
Its curvature is the $\mathfrak{g}$-valued $2$-form
\[
F = dA + \frac{1}{2}[A,A],
\]
equivalently $F=dA + A\wedge A$. Under a gauge transformation $g:M\to G$ one has
$A\mapsto g^{-1}Ag + g^{-1}dg$ and $F\mapsto g^{-1}Fg$.

\begin{definition}\label{def:quantum-curvature}
Let $A$ be a quantum connection on $X$ with values in $\End(V)$ and eccentricity $e\neq 0$.
The quantum curvature of $A$ is the family of local operators
\[
(F_A)_i:\M_{\mathrm{gen}}(U_i,\End(V))\dashrightarrow \M_{\mathrm{gen}}(U_i,\End(V))
\]
defined by
\[
(F_A)_i(f) := A_i(f)' - \frac{1}{2e}A_i(f)^2.
\]
\end{definition}

\begin{lemma}
Let $A$ be a quantum connection of eccentricity $e\neq 0$ with values in $\End(V)$.
On an overlap $U_i\cap U_j$ with change of coordinate $\lambda=\lambda_{ij}$, one has
\[
(F_A)_j = (\lambda')^2\bigl((F_A)_i\circ \lambda\bigr) + e\,S(\lambda)I,
\]
where $S(\lambda)$ is the Schwarzian derivative of $\lambda$, and $I$ is the identity endomorphism.
\end{lemma}

\begin{proof}
By the defining transformation rule for a quantum connection,
\[
A_j \;=\; \lambda'\,(A_i\circ \lambda)\;+\; e\,\frac{\lambda''}{\lambda'}\,I.
\]
Differentiating (with respect to the local coordinate on $U_j$) gives
\[
A_j'
\;=\;
\lambda''(A_i\circ\lambda) \;+\; (\lambda')^2(A_i'\circ\lambda)
\;+\; e\Bigl(\frac{\lambda''}{\lambda'}\Bigr)' I.
\]
Moreover, since $\frac{\lambda''}{\lambda'}I$ is scalar, it commutes with everything, and
\[
A_j^2
=
(\lambda')^2(A_i^2\circ\lambda)
+
2e\,\lambda''(A_i\circ\lambda)
+
e^2\Bigl(\frac{\lambda''}{\lambda'}\Bigr)^2 I.
\]
Substituting into the definition $(F_A)_j = A_j' - \frac{1}{2e}A_j^2$, the
$\lambda''(A_i\circ\lambda)$ terms cancel, yielding
\[
(F_A)_j
=
(\lambda')^2\Bigl( A_i' - \frac{1}{2e}A_i^2 \Bigr)\circ\lambda
\;+\;
e\left[\Bigl(\frac{\lambda''}{\lambda'}\Bigr)' - \frac12\Bigl(\frac{\lambda''}{\lambda'}\Bigr)^2\right]I.
\]
The bracketed expression is exactly $S(\lambda)$, so the claim follows. \qedhere
\end{proof}

\begin{remark}[Schwarzian as curvature]
Consider the scalar quantum $1$-differential $\Psi(f)=f'$. On the open set where $f'\neq 0$, the Maurer--Cartan construction yields the scalar quantum connection
$A(f)=f''/f'$. Its quantum curvature is
\[
F_A(f) = \left(\frac{f''}{f'}\right)' - \frac{1}{2}\left(\frac{f''}{f'}\right)^2 = S(f).
\]
Applying the previous lemma gives the chain rule
\[
S(f\circ \lambda) = (\lambda')^2\bigl(S(f)\circ \lambda\bigr) + S(\lambda).
\]
Since $S\!\left(\frac{az+b}{cz+d}\right)=0$ for every M\"obius transformation ($ad-bc\neq 0$), it follows that
$S\!\left(\frac{af+b}{cf+d}\right)=S(f)$ for every nonconstant meromorphic function $f$.
\end{remark}

\subsection*{Gauge action.}

\begin{definition}[Gauge group]\label{def:gauge-group}
Let $V$ be a finite-dimensional complex vector space. The gauge group $G$ consists of all classical invertible $0$-differentials with values in $\End(V)$,
i.e.\ meromorphic functions $g\in\mathcal{M}(X,\End(V))$ whose determinant is not identically zero.
\end{definition}

\begin{definition}[Gauge action]\label{def:gauge-action}
Let $A$ be a quantum connection with values in $\End(V)$ and eccentricity $e$. For $g\in G$ define
\[
g\bullet A := gAg^{-1} + g'g^{-1},
\]
i.e.\ locally $(g\bullet A)_i = g_iA_ig_i^{-1} + g'_ig_i^{-1}$.
\end{definition}

We recall that the appearance of an affine ``connection term'' in the transformation law
is classical; see Klein \cite[Chapter~III, \S7]{L32} for the corresponding discussion in the genus one setting.
\begin{proposition}
If $A$ is a quantum connection of eccentricity $e$, then $g\bullet A$ is again a quantum connection of eccentricity $e$.
Moreover, $(g,A)\mapsto g\bullet A$ defines an action of $G$.
\end{proposition}
\noindent
In particular, the curvature defined below is gauge covariant, while under change of coordinates it transforms
with a Schwarzian projective anomaly.

\subsection*{Vector bundles on Riemann surfaces}

Let $E\to X$ be a holomorphic vector bundle of rank $n$. Choose, for each $i$,
a holomorphic frame of $E$ over $U_i$, so that $E|_{U_i}\simeq U_i\times\mathbb{C}^n$.
On an overlap $U_i\cap U_j$ the corresponding transition function is a holomorphic map
\[
g_{ij}:U_i\cap U_j\longrightarrow \GL_n(\mathbb{C}),
\]
and a meromorphic section $s$ of $E$ is represented in these frames by meromorphic
functions $s_i\in \mathcal{M}(U_i,\mathbb{C}^n)$ satisfying
\[
s_j = g_{ji}\,(s_i\circ \lambda_{ij})
\qquad\text{on }U_i\cap U_j,
\]
where $g_{ji}=g_{ij}^{-1}$ and the composition with $\lambda_{ij}$ expresses the same point
in $U_i$-coordinates.

Recall the gauge action notation from Definition~\ref{def:gauge-action}: for an invertible matrix-valued meromorphic function $g$ on $U_i\cap U_j$ we write
\[
g\bullet B := g\,B\,g^{-1} + g'g^{-1},
\]
for the usual gauge action on local operators (where $g$ acts by pointwise multiplication).

A \emph{quantum connection on $E$} of eccentricity $e\in\mathbb{C}$ is a family of local operators
\[
A_i:\M_{\mathrm{gen}}(U_i,\mathbb{C}^n)\dashrightarrow \M_{\mathrm{gen}}(U_i,\mathbb{C}^n)
\]
such that on every overlap $U_i\cap U_j$ one has, as operators on $\M_{\mathrm{gen}}(U_i\cap U_j,\mathbb{C}^n)$,
\[
A_j
=
g_{ji}\bullet\Bigl(\lambda'_{ij}(A_i\circ \lambda_{ij})+e\frac{\lambda''_{ij}}{\lambda'_{ij}}I\Bigr).
\]
Equivalently,
\[
A_j
=
g_{ji}\Bigl(\lambda'_{ij}(A_i\circ \lambda_{ij})+e\frac{\lambda''_{ij}}{\lambda'_{ij}}I\Bigr)g_{ij}
\;+\;
(g_{ji})'g_{ij}.
\]
Here $I$ is the identity operator, the matrices $g_{ij}$ act by pointwise multiplication,
and primes denote differentiation with respect to the local coordinate on $U_j$.

Similarly, for $m\in\mathbb{Z}$, a \emph{quantum $m$-differential on $E$} is a family of local operators
\[
\Psi_i:\M_{\mathrm{gen}}(U_i,\mathbb{C}^n)\dashrightarrow \M_{\mathrm{gen}}(U_i,\mathbb{C}^n)
\]
such that on overlaps
\[
\Psi_j = (\lambda'_{ij})^{m}\,g_{ji}\,(\Psi_i\circ \lambda_{ij})\,g_{ij}.
\]
More generally, if $E$ and $F$ are holomorphic vector bundles with transition functions
$g^{E}_{ij}$ and $g^{F}_{ij}$, one may allow operators
$\Psi_i:\M_{\mathrm{gen}}(U_i,\mathbb{C}^{\mathrm{rk}\,E})\dashrightarrow \M_{\mathrm{gen}}(U_i,\mathbb{C}^{\mathrm{rk}\,F})$
transforming by
\[
\Psi_j = (\lambda'_{ij})^{m}\,g^{F}_{ji}\,(\Psi_i\circ \lambda_{ij})\,g^{E}_{ij}.
\]
In the trivial-bundle case $E=X\times V$ (so $g_{ij}=I$), these definitions reduce to
Definitions~\ref{def:quantum-connection} and~\ref{def:quantum-m-differential}.

\paragraph{Example (Multiplication by an $\End(E)$-valued $m$-differential).}
Let $\omega$ be a classical meromorphic section of $K_X^{m}\otimes \End(E)$.
In local coordinates and frames it is given by meromorphic matrices $\omega_i$ satisfying
\[
\omega_j = (\lambda'_{ij})^{m}\,g_{ji}\,(\omega_i\circ \lambda_{ij})\,g_{ij}.
\]
Then multiplication by $\omega$,
\[
\Psi_i(f):=\omega_i f,
\]
defines a (classical) quantum $m$-differential on $E$.

\paragraph{Example (Classical connections on $E\otimes K_X^{-e}$).}
Assume $e\in\mathbb{Z}$ and suppose $E\otimes K_X^{-e}$ is equipped with a meromorphic connection.
In local coordinates and frames, it is represented by meromorphic matrices $a_i$ whose
transformation rule is precisely the coefficient version of the above formula for $A_j$.
Equivalently, the $\mathcal{O}_X$-linear operators $A_i(f):=a_i f$ form a bundle-valued quantum
connection on $E$ of eccentricity $e$.

\subsection*{The notion of equivalence of two connections.}

\begin{definition}\label{def:equivalence}
Let $A_1$ and $A_2$ be quantum connections on $X$ with the same eccentricity $e\neq 0$ and values in $\End(V)$.
We say that $A_1$ and $A_2$ are classically equivalent if there exists $g\in G$ such that
\[
F_{A_2} = gF_{A_1}g^{-1}.
\]
\end{definition}

\begin{remark}
Although $F_A$ is not globally tensorial, it transforms on overlaps by
\[
(F_A)_j = (\lambda')^2\bigl((F_A)_i\circ \lambda\bigr) + eS(\lambda)I.
\]
The anomaly term $eS(\lambda)I$ is scalar (proportional to $I$). Since elements of $G$
act by conjugation and scalar endomorphisms commute with all elements of $\End(V)$,
the relation $F_{A_2} = gF_{A_1}g^{-1}$ is independent of the choice of local coordinates.
Thus the definition is well defined.
\end{remark}

\section{Second order quantum ODE}\label{sec:quantum-second-order-ode}\label{sec:classical-quantum-second-order-ode}
We discuss second order quantum differential equations on a Riemann surface.
The starting point is the local scalar equation on a domain $U\subset\Cc$,
\[
y''=p\,y'+q\,y,
\]
whose coefficients depend on the chosen coordinate. In the quantum setting we take the
\emph{solution} to be a quantum $0$-differential $\Psi$, and we allow the coefficients to be
quantum tensors (connections and $2$-differentials) so that the notion becomes intrinsic on
$X$.

More precisely, on each chart $U_i$ and for each meromorphic function $f$ on $U_i$ we obtain a
meromorphic function $\Psi_i(f)$ (defined generically on a Zariski open subset) which satisfies
a second order ODE whose coefficients depend on $f$:
\[
\Psi_i(f)''=p_i(f)\,\Psi_i(f)'+q_i(f)\,\Psi_i(f).
\]
We first determine the transformation law of the coefficient operators $p_i$ and $q_i$ under
change of coordinates.

\begin{lemma}\label{lem:quantum-second-order-coordinates}
Let $\Psi$ be a quantum $0$-differential and suppose that on each chart $U_i$
it satisfies the local equation
\[
\Psi_i(f)'' = p_i(f)\,\Psi_i(f)' + q_i(f)\,\Psi_i(f),
\qquad f \in \M_{\mathrm{gen}}(U_i),
\]
for some operators $p_i$ and $q_i$.
Then on an overlap $U_i\cap U_j$ with change of coordinate
$\lambda=\lambda_{ij}$, the coefficients transform according to
\[
 p_j
 =
 \lambda'\,(p_i\circ \lambda)
 +
 \frac{\lambda''}{\lambda'}\,I,
 \qquad
 q_j
 =
 (\lambda')^2\,(q_i\circ \lambda),
\]
as operators on $\M_{\mathrm{gen}}(U_i\cap U_j)$.
In particular, $q$ is a quantum $2$-differential, while $p$ transforms as a
quantum connection of eccentricity $1$.
\end{lemma}

\begin{proof}
All computations are performed on a Zariski open subset of $U_i\cap U_j$ on
which the relevant expressions are defined.
Since $\Psi$ is a quantum $0$-differential, on $U_i\cap U_j$ we have
\[
\Psi_j = \Psi_i \circ \lambda,
\qquad
\Psi_j(f)=\bigl(\Psi_i(f\circ \lambda^{-1})\bigr)\circ \lambda.
\]
Fix $f\in\M(U_j)$ and set $g=f\circ \lambda^{-1}\in\M(U_i)$
and $Y:=\Psi_i(g)$. Then
\[
\Psi_j(f)=Y\circ\lambda.
\]
Differentiating with respect to the $U_j$-coordinate gives
\[
\Psi_j(f)'=(Y'\circ\lambda)\lambda',
\qquad
\Psi_j(f)''
=
(Y''\circ\lambda)(\lambda')^2
+
(Y'\circ\lambda)\lambda''.
\]
Using the equation on $U_i$,
\[
Y''=p_i(g)Y'+q_i(g)Y,
\]
we obtain
\[
\begin{aligned}
\Psi_j(f)''
&=
(\lambda')^2\,(p_i(g)\circ\lambda)(Y'\circ\lambda)
+
(\lambda')^2\,(q_i(g)\circ\lambda)(Y\circ\lambda)
+
(Y'\circ\lambda)\lambda''.
\end{aligned}
\]
Since $Y\circ\lambda=\Psi_j(f)$ and
$Y'\circ\lambda = \Psi_j(f)'/\lambda'$, we conclude
\[
\Psi_j(f)''
=
\Bigl(
\lambda'(p_i(g)\circ\lambda)
+
\frac{\lambda''}{\lambda'}
\Bigr)\Psi_j(f)'
+
(\lambda')^2(q_i(g)\circ\lambda)\,\Psi_j(f).
\]
Replacing $g$ by $f\circ\lambda^{-1}$ yields the stated operator
transformation laws.
\end{proof}

\begin{remark}
By Lemma~\ref{lem:quantum-second-order-coordinates}, the coefficient $p$
transforms as a quantum connection of eccentricity $1$. Hence we may view
$p$ itself as a connection and form the corresponding covariant derivative
$\nabla_p$ acting on quantum differentials.

Since $\Psi$ is a quantum $0$-differential, one has
\[
(\nabla_p \Psi)(f)=\Psi(f)',
\]
and therefore
\[
(\nabla_p^2 \Psi)(f)
=
(\nabla_p(\nabla_p \Psi))(f)
=
\Psi(f)''-p(f)\Psi(f)'.
\]
Consequently, the local equation
\[
\Psi_i(f)''=p_i(f)\Psi_i(f)'+q_i(f)\Psi_i(f)
\]
may be rewritten intrinsically as
\[
\nabla_p^2 \Psi = q\,\Psi,
\]
which is independent of the choice of local coordinate.
\end{remark}

\begin{example}[A basic quantum second-order equation]
\label{ex:quantum-second-order-ode}
Let $U_i$ be a chart with coordinate $z$, and let
\[
A(f)=\frac{f''}{f'}
\]
be the scalar quantum connection (of eccentricity $1$) on $U_i$.
Let
\[
B(f)=(f')^2
\]
be the quantum $2$-differential given by squaring the first derivative.
Fix a constant $\omega\in\mathbb C$, and set
\[
p=A,
\qquad
q=\omega^2 B.
\]
Consider the quantum second-order equation
\[
\Psi_i(f)'' = p_i(f)\,\Psi_i(f)' + q_i(f)\,\Psi_i(f),
\qquad f\in \M(U_i).
\]
Writing $Y(z)=\Psi_i(f)(z)$, the equation becomes
\[
Y'' - \frac{f''}{f'}Y' - \omega^2 (f')^2 Y = 0.
\]
Set $t=f(z)$ and write $Y(z)=\widetilde Y(t)$. Then
\[
Y' = f' \widetilde Y',
\qquad
Y'' = f'' \widetilde Y' + (f')^2 \widetilde Y'',
\]
so the equation reduces to
\[
(f')^2\bigl(\widetilde Y'' - \omega^2 \widetilde Y\bigr)=0.
\]
Since $f'\neq 0$ on the domain of definition, we obtain the constant--coefficient
ODE
\[
\widetilde Y'' = \omega^2 \widetilde Y.
\]
Thus the general local solution is
\[
\widetilde Y(t)=C_1 e^{\omega t}+C_2 e^{-\omega t},
\]
and hence
\[
\Psi_i(f)(z)
=
C_1 e^{\omega f(z)} + C_2 e^{-\omega f(z)}.
\]
\end{example}
In Part~V we interpret the above differential equation as a (quantum) mass--spring system with frequency $\omega$.
\subsection*{Gauge invariance}

Unfortunately the quantum curvature $F_A$ of a quantum connection is not gauge
invariant. This defect becomes particularly transparent when the eccentricity
is $e=\tfrac12$, since in that case we may consider the second order quantum ODE
\[
\Psi'' = 2A\Psi' + q\Psi,
\]
where $\Psi$ is a quantum $0$-differential and $q$ is a quantum $2$-differential.
Equivalently, since for $e=\tfrac12$ and $m=0$ one has
\[
(\nabla_A^2\Psi)(f)=\Psi(f)''-2A(f)\Psi(f)',
\]
the equation may be written intrinsically as $\nabla_A^2\Psi=q\Psi$.
\subsection{Main Lemma}
We are now ready to state our main lemma. Suppose we have an \emph{invertible} solution $g$ of this
equation with values in $\End(V)$, then the corresponding gauge transform of $A$
has curvature $F_A-q$ up to conjugation.

\begin{lemma}\label{lem:gauge-curvature-ehalf}
Let $A$ be a quantum connection of eccentricity $e=\tfrac12$ with values in
$\End(V)$, and let $q$ be an $\End(V)$-valued quantum $2$-differential.
Let $g$ be an $\End(V)$-valued quantum $0$-differential such that for every chart
$U_i$ and every $f\in \M(U_i)$ the meromorphic function
$g_i(f)\in \M(U_i,\End(V))$ satisfies $\det(g_i(f))\not\equiv 0$ on $U_i$.
Assume that for every $i$ and every $f\in \M(U_i)$ one has, on the open subset
where $g_i(f)$ is pointwise invertible,
\[
g_i(f)'' \;=\; 2A_i(f)\,g_i(f)' \;+\; q_i(f)\,g_i(f).
\]

Consider the gauge-transformed connection $g^{-1}\bullet A$ (in the sense of
Section~3). Explicitly, on $U_i$ and for $f\in\M(U_i)$ one has
\[
(g^{-1}\bullet A)_i(f)
\;=\;
g_i(f)^{-1}A_i(f)g_i(f)\;-\;g_i(f)^{-1}g_i(f)'.
\]
Then the quantum curvature satisfies, pointwise in $f$,
\[
F_{g^{-1}\bullet A,\, i}(f)
\;=\;
g_i(f)^{-1}\bigl(F_{A,\, i}(f)-q_i(f)\bigr)g_i(f),
\]
where for $e=\tfrac12$ one has $F_{A,\, i}(f)=A_i(f)'-A_i(f)^2$.
\end{lemma}

\begin{proof}
Fix a chart $U_i$ and an input $f\in \M(U_i)$. On the open subset where
$g:=g_i(f)$ is pointwise invertible, write for brevity
\[
A:=A_i(f),\qquad q:=q_i(f),\qquad 
\widetilde A := (g^{-1}\bullet A)_i(f)=g^{-1}Ag-g^{-1}g'.
\]
Since $e=\tfrac12$, the curvature is $F_A=A'-A^2$, and hence
\[
F_{\widetilde A}=(\widetilde A)'-\widetilde A^{\,2}.
\tag{1}
\]

We first expand the derivative term. Using $(g^{-1})'=-g^{-1}g'g^{-1}$, we obtain
\[
\begin{aligned}
(\widetilde A)'
&=(g^{-1}Ag-g^{-1}g')' \\
&= -g^{-1}g'g^{-1}Ag + g^{-1}A'g + g^{-1}Ag'
   + g^{-1}g'g^{-1}g' - g^{-1}g''.
\end{aligned}
\tag{2}
\]
Next we expand the square:
\[
\begin{aligned}
\widetilde A^{\,2}
&=(g^{-1}Ag-g^{-1}g')^2 \\
&= g^{-1}A^2g + g^{-1}g'g^{-1}g' - g^{-1}Ag' - g^{-1}g'g^{-1}Ag .
\end{aligned}
\tag{3}
\]
Subtracting (3) from (2) gives
\[
F_{\widetilde A}
= g^{-1}A'g-g^{-1}A^2g - g^{-1}(g''-2Ag').
\tag{4}
\]
By hypothesis, $g$ satisfies $g''=2Ag'+qg$, so $(g''-2Ag')=qg$ and therefore
\[
F_{\widetilde A}=g^{-1}(A'-A^2)g-g^{-1}qg = g^{-1}(F_A-q)g.
\tag{5}
\]
This proves the claim.
\end{proof}

\subsection*{The quantum Schwarzian.}
In the scalar case, Proposition~\ref{prop:classical-schwarzian-ode} shows that the
Schwarzian of the ratio of two independent solutions is
\[
S\Bigl(\frac{y_1}{y_2}\Bigr)=2\bigl(p'-p^2-q\bigr)=2(F_p-q).
\]
This motivates the following definition in the non-abelian quantum setting.

\begin{definition}\label{def:quantum-schwarzian}
Let $A$ be a quantum connection of eccentricity $e=\tfrac12$ with values in
$\End(V)$, and let $q$ be an $\End(V)$-valued quantum $2$-differential.
For the quantum second-order equation
\begin{equation}\label{eq:quantum-second-order-ehalf}
\Psi''=2A\Psi' + q\Psi
\qquad\bigl(\text{equivalently, }\nabla_A^2\Psi=q\Psi\bigr),
\end{equation}
we define its \emph{quantum Schwarzian} to be the $\End(V)$-valued quantity
\[
\mathcal{S}_{A,q}:=2\,(F_A-q).
\]
\end{definition}

\begin{remark}
The quantum Schwarzian is the natural analogue of the classical invariant in
Proposition~\ref{prop:classical-schwarzian-ode}.
On an overlap $U_i\cap U_j$ with change of coordinate $\lambda=\lambda_{ij}$,
the curvature of an eccentricity $e=\tfrac12$ connection satisfies the
transformation law
\[
F_{A,j}=(\lambda')^2\,(F_{A,i}\circ\lambda)+\tfrac12 S(\lambda)\,I,
\]
while $q$ transforms tensorially as a quantum $2$-differential. Hence
\[
\mathcal{S}_{A,q,j}=(\lambda')^2\,(\mathcal{S}_{A,q,i}\circ\lambda)+S(\lambda)\,I.
\]
In particular, $\mathcal{S}_{A,q}$ has the \emph{same anomaly} under coordinate
change as the curvature (up to the normalizing factor~$2$).

Lemma~\ref{lem:gauge-curvature-ehalf} explains the gauge--theoretic meaning of
$\mathcal{S}_{A,q}$. Namely, if $g$ is an invertible $\End(V)$-valued solution of
\eqref{eq:quantum-second-order-ehalf}, then
\[
2F_{g^{-1}\bullet A}=g^{-1}\,\mathcal{S}_{A,q}\,g.
\]
Thus the quantum Schwarzian is precisely the curvature of the gauge-transformed
connection $g^{-1}\bullet A$, up to conjugation. In the scalar case
($V=\mathbb{C}$) this reduces to gauge invariance, and it is the direct quantum
analogue of the fact that $S(y_1/y_2)$ is independent of the choice of a basis of
solutions.
\end{remark}

\section{The gauge and $\star$ actions}\label{sec:star-action}

Let $A$ be a quantum connection of eccentricity $e=\tfrac12$ with values in
$\End(V)$, and let $q$ be an $\End(V)$-valued quantum $2$-differential. The
second-order equation \eqref{eq:quantum-second-order-ehalf} depends on the choice
of the pair $(A,q)$, while its quantum Schwarzian $\mathcal{S}_{A,q}=2(F_A-q)$ is
a more intrinsic quantity. We now describe a natural action of quantum
$1$-differentials on such pairs which preserves $F_A-q$ (and hence
$\mathcal{S}_{A,q}$).

Compare the scalar Miura transformation for affine and projective connections in
Frenkel~\cite{L37}, \S8.5, pp.~95--96. To the best of our knowledge, however,
the following matrix-valued affine $\star$-action on pairs $(A,q)$ is new.

\begin{definition}\label{def:star-action}
Let $u$ be an $\End(V)$-valued quantum $1$-differential on $X$. We define
\[
u\star (A,q):=(A+u,\; q-uA-Au-u^2+u'),
\]
where the expression is understood chartwise: on $U_i$ and for
$f\in \M_{\mathrm{gen}}(U_i)$,
\[
(u\star(A,q))_i(f)=\bigl(A_i(f)+u_i(f),\; q_i(f)-u_i(f)A_i(f)-A_i(f)u_i(f)-u_i(f)^2+u_i(f)'\bigr),
\]
and all identities are interpreted on a common Zariski open subset on which the
relevant expressions are defined.
\end{definition}

\begin{lemma}\label{lem:star-is-action}
The operation $\star$ defines an action of the additive group of $\End(V)$-valued
quantum $1$-differentials on the set of pairs $(A,q)$ with $e$.
Equivalently, for quantum $1$-differentials $u$ and $v$ one has
\[
v\star(u\star(A,q))=(u+v)\star(A,q).
\]
\end{lemma}

\begin{proof}
This is a direct computation using bilinearity of the products and linearity of
the derivative. Writing $(\widetilde A,\widetilde q):=u\star(A,q)$, we have $\widetilde A=A+u$ and
$\widetilde q=q-uA-Au-u^2+u'$. Hence
\[
\begin{aligned}
v\star(\widetilde A,\widetilde q)
&=\bigl(A+u+v,\; q-uA-Au-u^2+u' -v(A+u)-(A+u)v-v^2+v'\bigr) \\
&=\bigl(A+u+v,\; q-(u+v)A-A(u+v)-(u+v)^2+(u+v)'\bigr) \\
&=(u+v)\star(A,q),
\end{aligned}
\]
as claimed.
\end{proof}

\begin{lemma}\label{lem:FA-q-invariant-star}
For $e=\tfrac12$ the quantum Schwarzian is invariant under the $\star$-action. In
other words,
\[
\mathcal{S}_{A,q}=\mathcal{S}_{u\star(A,q)}.
\]
\end{lemma}

\begin{proof}
For $e=\tfrac12$ one has $F_A=A'-A^2$, hence
\[
F_{A+u}=(A+u)'-(A+u)^2 = F_A+u'-Au-uA-u^2.
\]
Subtracting the transformed $q$ gives
\[
F_{A+u}-\bigl(q-uA-Au-u^2+u'\bigr)=F_A-q,
\]
as desired.
\end{proof}

\subsection{Main Theorem}
\begin{theorem}[Main Theorem]
\label{lem:gauge-action-on-q}
Let $g\in G$ be a gauge transformation (Definition~\ref{def:gauge-group}). Let $A$
be an $\End(V)$-valued quantum connection of eccentricity $e=\tfrac12$ and let $q$
be an $\End(V)$-valued quantum $2$-differential. Define
\[
A^{g}:= g^{-1}\bullet A = g^{-1}Ag - g^{-1}g',
\qquad
q^{g}:= g^{-1}qg + 2\,g^{-1}A g' - g^{-1}g''.
\]
Then $A^{g}$ is again a quantum connection of eccentricity $e=\tfrac12$, and $q^{g}$
is an $\End(V)$-valued quantum $2$-differential. Moreover, if $\Psi$ is a (quantum)
$0$-differential solving
\[
\Psi'' = 2A\Psi' + q\Psi,
\]
then $\widetilde\Psi:=g^{-1}\Psi$ solves the transformed equation
\[
\widetilde\Psi'' = 2A^{g}\widetilde\Psi' + q^{g}\widetilde\Psi.
\]
In particular,
\[
F_{A^{g}}-q^{g} = g^{-1}(F_A-q)g,
\qquad\text{hence}\qquad
\mathcal{S}_{A^{g},q^{g}} = g^{-1}\mathcal{S}_{A,q}g.
\]
Consequently the coefficients of the characteristic polynomial of $\mathcal{S}_{A,q}$
(equivalently, the functions $\operatorname{tr}(\mathcal{S}_{A,q}^r)$) are invariants
of the gauge-equivalence class of the equation.
In particular, if $g$ itself is a (classical) $\GL(V)$-valued solution of
\(
 g''=2Ag'+qg,
\)
then $q^{g}=0$ and $F_{A^{g}}=g^{-1}(F_A-q)g$, recovering
Lemma~\ref{lem:gauge-curvature-ehalf}.
\end{theorem}

\begin{proof}
The first statement is immediate from the fact that the gauge action preserves the
class of quantum connections (apply Proposition following
Definition~\ref{def:gauge-action} to $g^{-1}$).
For the transformation of the equation, write $\Psi=g\widetilde\Psi$. Then
\[
\Psi' = g'\widetilde\Psi + g\widetilde\Psi',
\qquad
\Psi'' = g''\widetilde\Psi + 2g'\widetilde\Psi' + g\widetilde\Psi''.
\]
Substituting into $\Psi''=2A\Psi'+q\Psi$ and multiplying on the left by $g^{-1}$
yields
\[
\widetilde\Psi'' = 2(g^{-1}Ag-g^{-1}g')\,\widetilde\Psi' +
\bigl(g^{-1}qg + 2g^{-1}Ag' - g^{-1}g''\bigr)\,\widetilde\Psi,
\]
which is the claimed equation.

Finally, since $e=\tfrac12$ we have $F_A=A'-A^2$. A direct computation from
$A^{g}=g^{-1}Ag-g^{-1}g'$ gives
\[
F_{A^{g}} = g^{-1}F_A g + 2\,g^{-1}A g' - g^{-1}g'',
\]
and subtracting $q^{g}$ yields $F_{A^{g}}-q^{g}=g^{-1}(F_A-q)g$. Multiplying by
$2$ gives the stated transformation of $\mathcal{S}_{A,q}$.
Conjugation invariance of characteristic polynomial coefficients and traces is
immediate. 
\end{proof}

\begin{remark}\label{rem:curvature-not-gauge-invariant}
It is important that the transformation of $q$ in Theorem~\ref{lem:gauge-action-on-q}
is \emph{not} given by tensorial conjugation alone. In general the curvature
$F_A=A'-A^2$ is \emph{not} gauge covariant under the action
$A\mapsto g\bullet A=gAg^{-1}+g'g^{-1}$ when $A$ is $\End(V)$-valued and noncommutative.
The additional terms $2g^{-1}Ag'-g^{-1}g''$ in $q^{g}$ are exactly what is needed so
that the \emph{Schwarzian combination} $F_A-q$ transforms by conjugation, and hence
$\mathcal{S}_{A,q}=2(F_A-q)$ is the natural gauge-theoretic invariant of the second
order equation.

\end{remark}

\begin{lemma}[Compatibility of the gauge and $\star$ actions]\label{lem:gauge-star-compatible}
Let $A$ be an $\End(V)$-valued quantum connection of eccentricity $e=\tfrac12$ and
let $q$ be an $\End(V)$-valued quantum $2$-differential. Let $u$ be an
$\End(V)$-valued quantum $1$-differential and let $g\in G$ be a gauge
transformation. Set $u^{g}:=g^{-1}ug$, and let $(A^{g},q^{g})$ denote the gauge
transform of $(A,q)$ as in Theorem~\ref{lem:gauge-action-on-q}. Then the $\star$
and gauge actions are compatible in the sense that
\[
(u\star(A,q))^{g} \;=\; u^{g}\star(A^{g},q^{g}),
\]
where $(\cdot)^{g}$ denotes the gauge transformation of pairs given by
Theorem~\ref{lem:gauge-action-on-q}. Equivalently, gauge transformations send
$\star$-orbits of pairs $(A,q)$ to $\star$-orbits.
\end{lemma}

\begin{proof}
Write $u\star(A,q)=(\widehat A,\widehat q)$, so $\widehat A=A+u$ and
$\widehat q=q-uA-Au-u^2+u'$. The first components agree:
\[
\widehat A^{g}=g^{-1}(A+u)g-g^{-1}g'=(g^{-1}Ag-g^{-1}g')+g^{-1}ug=A^{g}+u^{g}.
\]
For the second component, it is convenient to use the gauge-covariant quantity
$F_A-q$. By Lemma~\ref{lem:FA-q-invariant-star} one has
$F_{\widehat A}-\widehat q=F_A-q$, hence Theorem~\ref{lem:gauge-action-on-q} gives
\[
F_{\widehat A^{g}}-\widehat q^{g}=g^{-1}(F_A-q)g.
\]
On the other hand, Theorem~\ref{lem:gauge-action-on-q} implies
$F_{A^{g}}-q^{g}=g^{-1}(F_A-q)g$. Apply Lemma~\ref{lem:FA-q-invariant-star} to the
pair $(A^{g},q^{g})$ and the $1$-differential $u^{g}$ to write
$u^{g}\star(A^{g},q^{g})=(A^{g}+u^{g},\widehat q_g)$ with
\[
F_{A^{g}+u^{g}}-\widehat q_g=F_{A^{g}}-q^{g}=g^{-1}(F_A-q)g.
\]
Since $\widehat A^{g}=A^{g}+u^{g}$, comparing the two displayed identities shows
$\widehat q^{g}=\widehat q_g$, which proves $(u\star(A,q))^{g}=u^{g}\star(A^{g},q^{g})$.
\end{proof}

\begin{remark}
The action $\star$ is transitive on the choice of the connection component: if
$A_1$ and $A_2$ are two eccentricity $e=\tfrac12$ quantum connections, then
$u:=A_2-A_1$ is a quantum $1$-differential and $A_2=A_1+u$. Thus $\star$
parametrizes different presentations $(A,q)$ of the same quantum Schwarzian.
In the classical scalar case this recovers the familiar fact that a projective
connection can be represented by an affine connection modulo addition of a
$1$-form.
\end{remark}

\begin{theorem}\label{lem:second-order-equivalence}
Let $A$ be a quantum connection of eccentricity $e=\tfrac12$ with values in
$\End(V)$, and let $q$ be an $\End(V)$-valued quantum $2$-differential.
Suppose $g$ is an $\End(V)$-valued quantum $0$-differential such that
$\det(g_i(f))\not\equiv 0$ for every chart $U_i$ and every $f\in \M(U_i)$, and
assume that $g$ satisfies the second-order equation
\[
\nabla_A^2 g = qg
\qquad
\bigl(\text{equivalently, } g'' = 2Ag' + qg \bigr).
\]
Then for any two invertible solutions $g_1,g_2$ of this equation, the
connections $g_1^{-1}\bullet A$ and $g_2^{-1}\bullet A$ are classically equivalent.
\end{theorem}

\begin{proof}
By Lemma~\ref{lem:gauge-curvature-ehalf}, for any solution $g$ one has
\[
F_{g^{-1}\bullet A} = g^{-1}(F_A-q)g.
\]
If $g_1$ and $g_2$ are solutions, set $h:=g_2^{-1}g_1$ on the open subset where
both are invertible. Then
\[
F_{g_1^{-1}\bullet A} = h^{-1}F_{g_2^{-1}\bullet A}h,
\]
which is precisely the condition of classical equivalence.
\end{proof}

\begin{corollary}\label{cor:charpoly-well-defined}
Let $A$ be a quantum connection of eccentricity $e=\tfrac12$ with values in
$\End(V)$, and let $q$ be an $\End(V)$-valued quantum $2$-differential.
Let $g_1$ and $g_2$ be two invertible (classical) $\mathrm{GL}(V)$-valued solutions of
\[
g''=2Ag'+qg.
\]
Then the connections $g_1^{-1}\bullet A$ and $g_2^{-1}\bullet A$ are equivalent
in the sense of Definition~\ref{def:equivalence}. In particular, on each chart $U_i$ the
coefficients of the characteristic polynomial of $F_{g^{-1}\bullet A,\,i}$
are independent of the choice of an invertible $\mathrm{GL}(V)$-valued solution
$g$ of $g''=2Ag'+qg$.

Moreover, since $F_{g^{-1}\bullet A}$ satisfies the same Schwarzian
coordinate--change anomaly as any quantum curvature, these coefficients define
invariants of the equation on $X$ in the same sense as the curvature (namely,
their local expressions transform consistently on overlaps).
\end{corollary}

\begin{proof}
Set $h:=g_2^{-1}g_1$ on the open subset where both solutions are pointwise
invertible. By Lemma~\ref{lem:second-order-equivalence} one has
\[
F_{g_1^{-1}\bullet A}=h^{-1}F_{g_2^{-1}\bullet A}h.
\]
By Definition~\ref{def:equivalence} this conjugacy of curvatures implies that the connections
$g_1^{-1}\bullet A$ and $g_2^{-1}\bullet A$ are equivalent.

Finally, $F_{g_1^{-1}\bullet A}$ and $F_{g_2^{-1}\bullet A}$ are pointwise
conjugate matrices, hence have the same characteristic polynomial. This proves
that the coefficients of the characteristic polynomial of $F_{g^{-1}\bullet A}$
do not depend on the choice of $g$.
\end{proof}

\begin{remark}\label{rem:quantum-characteristic-invariants}
Let $g$ be an invertible (classical) $\mathrm{GL}(V)$-valued solution of the
quantum second-order equation
\[
g''=2Ag'+qg,
\]
and set
\[
B:=F_{g^{-1}\bullet A}.
\]
By Corollary~\ref{cor:charpoly-well-defined}, the conjugacy class of $B$ is
independent of the choice of the solution $g$. Hence, for every integer
$1\le r\le \dim(V)$, the functions $\operatorname{tr}(B^r)$ (equivalently, the
coefficients of the characteristic polynomial of $B$) are well-defined
invariants of the differential equation, independent of $g$.

Moreover, Lemma~\ref{lem:gauge-curvature-ehalf} gives $B=g^{-1}(F_A-q)g$, and
therefore for every $1\le r\le \dim(V)$ one has
\[
\operatorname{tr}(B^r)=\operatorname{tr}\bigl((F_A-q)^r\bigr).
\]
In particular, these invariants are well defined on each chart and on $X$ (in
the same sense as the curvature): under a change of coordinate the curvature
acquires a Schwarzian anomaly, while $q$ transforms tensorially, so $F_A-q$ (and
hence $B$) has the same Schwarzian anomaly as $F_A$.

Following Chern, we call the quantities $\operatorname{tr}(B^r)$ (or,
equivalently, the coefficients of the characteristic polynomial of $B$) the
\emph{quantum characteristic invariants} of the equation.
\end{remark}

\part{Modularity}

\setcounter{section}{10}

\section{Modular forms and modular connections}\label{sec:modular-forms}

Any hyperbolic Riemann surface, and in particular any Riemann surface of genus greater than one, can be realized as a quotient
\[
X = \Gamma\backslash \Hh,
\]
where $\Hh$ is the upper half-plane and $\Gamma\subset \SL_2(\mathbb{R})$ is a Fuchsian subgroup acting properly discontinuously.

Recall that an element
\[
\gamma=\begin{pmatrix} a & b \\ c & d\end{pmatrix}\in \SL_2(\mathbb{R})
\]
acts on $z\in\Hh$ by the M\"obius transformation
\[
\gamma(z) = \frac{az+b}{cz+d}.
\]
Thus a meromorphic $m$-differential $f(z)(dz)^{m}$ on $X$ corresponds to a meromorphic function $f$ on $\Hh$ such that
\[
f(\gamma(z))\,(d\gamma(z))^{m} = f(z)\,dz^{m}\qquad (\gamma\in \Gamma).
\]
Since $d\gamma(z)=(cz+d)^{-2}dz$, we must have
\begin{equation}\label{eq:4.1}
f(\gamma z) = (cz+d)^{2m}f(z).
\end{equation}

Throughout this section, for a meromorphic function $h$ on $\Hh$ we write
$h':=Dh:=\frac{1}{2\pi i}\frac{dh}{dz}$. For a M\"obius transformation
$\gamma(z)=\frac{az+b}{cz+d}$ we continue to write $\gamma'(z)$ and $\gamma''(z)$
for the ordinary derivatives with respect to $z$ (so $\gamma'(z)=(cz+d)^{-2}$).

Motivated by \eqref{eq:4.1}, for any integer $k$ and complex vector space $V$ we set
\[
M_k(\Gamma,V):=
\left\{\psi:\Hh\to V\ \text{meromorphic} : \psi(\gamma z) = (cz+d)^{k}\psi(z)
\ \text{for all }\gamma\in\Gamma\right\}.
\]
We refer to elements of $M_k(\Gamma,V)$ as $V$-valued meromorphic modular forms
of weight $k$ for $\Gamma$.

In particular, weight $0$ forms are precisely $\Gamma$-invariant meromorphic functions, hence
\[
M_0(\Gamma,\mathbb{C})\cong \mathcal{M}(X),
\]
and more generally $M_0(\Gamma,V)$ identifies with meromorphic $V$-valued functions on $X$.

\begin{remark}[Classical $m$-differentials]
Let $V$ be finite-dimensional. A classical meromorphic $m$-differential on $X$ with values in $\End(V)$ may be represented on $\Hh$ as an $\End(V)$-valued meromorphic function $\phi$ such that $\phi(z)(dz)^{m}$ is $\Gamma$-invariant. Equivalently,
\[
\phi(\gamma z) = (cz+d)^{2m}\phi(z),\qquad \gamma\in \Gamma,
\]
so $\phi\in M_{2m}(\Gamma,\End(V))$. Conversely, every $\phi\in M_{2m}(\Gamma,\End(V))$ determines a classical $\End(V)$-valued $m$-differential on $X$.
\end{remark}

\subsection{Global quantum m-differentials}

\begin{definition}
A global quantum $m$-differential on $X=\Gamma\backslash \Hh$ with values in $V$ is a (linear) map
\[
\Psi : M_0(\Gamma,V)\longrightarrow M_{2m}(\Gamma,V).
\]
\end{definition}

This definition captures the basic idea that a quantum $m$-differential should send meromorphic functions on $X$ (weight $0$) to meromorphic $m$-differentials on $X$ (weight $2m$).

To recover the local notion of quantum $m$-differential from Section~2 using an atlas whose transition maps lie in $\Gamma$, one may impose the following stronger equivariance condition.

\begin{definition}
Let $m\in\mathbb{Z}$. For a linear operator
\[
\Psi:\mathcal{M}(\Hh,V)\longrightarrow \mathcal{M}(\Hh,V)
\]
and $\gamma\in\Gamma$, define the transformed operator
\[
(\Psi|_{m}\gamma)(f) := (\gamma')^{m}\bigl(\Psi(f\circ \gamma^{-1})\bigr)\circ \gamma.
\]
We say that $\Psi$ is $\Gamma$-equivariant of weight $m$ if
\[
\Psi|_{m}\gamma = \Psi \qquad \text{for all }\gamma\in\Gamma.
\]
\end{definition}

\begin{proposition}
If $\Psi$ is $\Gamma$-equivariant of weight $m$, then it induces a global quantum $m$-differential in the sense of the first definition, i.e.
\[
\Psi\bigl(M_0(\Gamma,V)\bigr)\subseteq M_{2m}(\Gamma,V).
\]
\end{proposition}

\begin{proof}
Let $f\in M_0(\Gamma,V)$, so $f\circ \gamma^{-1}=f$. Since $\Psi|_{m}\gamma=\Psi$, we have
\[
\Psi(f) = (\Psi|_{m}\gamma)(f) = (\gamma')^{m}\,\Psi(f)\circ \gamma.
\]
Equivalently,
\[
\Psi(f)\circ \gamma = (\gamma')^{-m}\Psi(f) = (cz+d)^{2m}\Psi(f),
\]
because $\gamma'(z)=(cz+d)^{-2}$ for $\gamma\in \SL_2(\mathbb{R})$. Thus $\Psi(f)\in M_{2m}(\Gamma,V)$. \qedhere
\end{proof}

\begin{remark}
In the classical situation, multiplication by an $\End(V)$-valued $m$-differential provides an example of the above. Indeed, if $\phi\in M_{2m}(\Gamma,\End(V))$, then
\[
(\Psi_{\phi}f)(z):=\phi(z)f(z)
\]
defines a $\Gamma$-equivariant operator of weight $m$, hence a global quantum $m$-differential.
\end{remark}
\subsection{Quantum modular connections}
Let $e\in\mathbb{C}$ and let $V$ be a finite-dimensional complex vector space. Recall that for
\[
\gamma=\begin{pmatrix}a&b\\ c&d\end{pmatrix}\in \SL_2(\mathbb{R})
\]
one has
\[
\gamma(z)=\frac{az+b}{cz+d},\qquad \gamma'(z)=\frac{1}{(cz+d)^{2}}.
\]

\begin{definition}
A quantum modular connection for $\Gamma$ with values in $V$ and eccentricity $e$ is a linear operator
\[
A:\mathcal{M}(\Hh,V)\longrightarrow \mathcal{M}(\Hh,V)
\]
such that for every $\gamma\in\Gamma$ one has
\[
A = A\|_{e}\gamma,
\]
where the transformed operator $A\|_{e}\gamma$ is defined by
\[
(A\|_{e}\gamma)(f) := \gamma'\bigl(A(f\circ \gamma^{-1})\bigr)\circ \gamma + \frac{e}{2\pi i}\frac{\gamma''}{\gamma'}f,\qquad f\in\mathcal{M}(\Hh,V).
\]
\end{definition}

\begin{remark}
If $A$ is a quantum modular connection and $f\in M_0(\Gamma,V)$ is $\Gamma$-invariant, then $A(f)$ transforms by the inhomogeneous rule
\[
A(f)(\gamma z) = (cz+d)^{2}A(f)(z) + \frac{2e}{2\pi i}c(cz+d)f(z),\qquad
\gamma=\begin{pmatrix}a&b\\ c&d\end{pmatrix}\in \Gamma,
\]
since $\gamma'(z)^{-1}=(cz+d)^2$ and $-\gamma''(z)/\gamma'(z)^{2}=2c(cz+d)$. Thus $A$ descends to a connection-type operator on $X=\Gamma\backslash \Hh$.
\end{remark}

\subsection{Classical modular connections}

\begin{definition}
A classical modular connection for $\Gamma$ with values in $\End(V)$ and eccentricity $e$ is a meromorphic function
\[
a:\Hh\longrightarrow \End(V)
\]
such that for every $\gamma=\begin{pmatrix}a&b\\ c&d\end{pmatrix}\in\Gamma$,
\[
a(\gamma z) = (cz+d)^{2}a(z) + \frac{2e}{2\pi i}c(cz+d)I,
\]
where $I\in \End(V)$ denotes the identity endomorphism.
\end{definition}

\begin{remark}
If $a_1$ and $a_2$ are classical modular connections of the same eccentricity $e$, then their difference is a weight $2$ matrix modular form:
\[
(a_2-a_1)(\gamma z) = (cz+d)^{2}(a_2-a_1)(z),
\]
so $a_2-a_1\in M_2(\Gamma,\End(V))$. Thus the space of classical modular connections of eccentricity $e$ is an affine space (a torsor) over $M_2(\Gamma,\End(V))$.
\end{remark}

\begin{proposition}
Let $a$ be a classical modular connection for $\Gamma$ with values in $\End(V)$ and eccentricity $e$. Define an operator
\[
A:\mathcal{M}(\Hh,V)\longrightarrow \mathcal{M}(\Hh,V),\qquad A(f):=a\cdot f.
\]
Then $A$ is a quantum modular connection in the sense of the quantum modular connection definition. Moreover, viewed on $X=\Gamma\backslash \Hh$, it is $\mathcal{O}_X$-linear.
\end{proposition}

\begin{proof}
Let $\gamma\in\Gamma$ and $f\in\mathcal{M}(\Hh,V)$. Since $A(f\circ \gamma^{-1}) = a(f\circ \gamma^{-1})$, we have
\[
\gamma'\bigl(A(f\circ \gamma^{-1})\bigr)\circ \gamma = \gamma'(a\circ \gamma)f.
\]
Hence
\[
(A\|_{e}\gamma)(f) = \bigl(\gamma'(a\circ \gamma) + \frac{e}{2\pi i}\frac{\gamma''}{\gamma'}I\bigr)f.
\]
Thus $A=A\|_{e}\gamma$ for all $\gamma$ is equivalent to
\[
a(z) = \gamma'(z)a(\gamma z) + e\frac{\gamma''(z)}{\gamma'(z)}I,
\]
which, after rewriting, is precisely the defining transformation rule for $a$.
The $\mathcal{O}_X$-linearity is immediate. \qedhere
\end{proof}

\subsection{Modular covariant derivative}

Let $e\neq 0$ and let $a:\Hh\to \End(V)$ be a classical modular connection for $\Gamma$ of eccentricity $e$.

\begin{definition}
For $\psi\in M_k(\Gamma,V)$ we define the modular covariant derivative associated to $a$ by
\[
\nabla_a(\psi) := \psi' - \frac{k}{2e}a\psi,
\]
where $' = D$ is the normalized derivation $D=\frac{1}{2\pi i}\frac{d}{dz}$ and $a$ acts on $V$ via the natural action of $\End(V)$.
\end{definition}

\begin{proposition}
The operator in the definition above is well-defined and satisfies
\[
\nabla_a : M_k(\Gamma,V)\longrightarrow M_{k+2}(\Gamma,V).
\]
\end{proposition}

\begin{proof}
Let $\gamma=\begin{pmatrix}a&b\\ c&d\end{pmatrix}\in\Gamma$. Since $\psi\in M_k(\Gamma,V)$, we have
\[
\psi(\gamma z) = (cz+d)^{k}\psi(z).
\]
Differentiating and using $\gamma'(z)=(cz+d)^{-2}$ yields
\[
\psi'(\gamma z) = (cz+d)^{k+2}\psi'(z) + \frac{k}{2\pi i}c(cz+d)^{k+1}\psi(z).
\]
On the other hand, the transformation rule for $a$ gives
\[
a(\gamma z)\psi(\gamma z) = (cz+d)^{k+2}a(z)\psi(z) + \frac{2e}{2\pi i}c(cz+d)^{k+1}\psi(z).
\]
Therefore
\[
(\nabla_a\psi)(\gamma z) = \psi'(\gamma z) - \frac{k}{2e}a(\gamma z)\psi(\gamma z)
= (cz+d)^{k+2}\left(\psi'(z) - \frac{k}{2e}a(z)\psi(z)\right),
\]
since the terms proportional to $c(cz+d)^{k+1}\psi(z)$ cancel. Thus $\nabla_a\psi\in M_{k+2}(\Gamma,V)$. \qedhere
\end{proof}

\subsection{The Serre derivative}

We now specialize to $\Gamma=\SL_2(\mathbb{Z})$. Using the normalized derivation
\[
D=\frac{1}{2\pi i}\frac{d}{dz},
\]
we write $\psi' := D\psi$ in this and the subsequent subsections (thus $d\psi/dz = 2\pi i\,\psi'$).

For $m\ge 2$, the Eisenstein series of weight $2m$ is
\[
E_{2m}(z) = 1 - \frac{4m}{B_{2m}}\sum_{n\ge 1}\sigma_{2m-1}(n)q^{n},
\qquad q=e^{2\pi i z},
\]
where $B_{2m}$ is the Bernoulli number and $\sigma_r(n)=\sum_{d\mid n}d^{r}$.
For $m=1$ the above series is no longer absolutely convergent; the holomorphic function
defined by the Fourier expansion
\[
E_2(z)=1-24\sum_{n\ge 1}\sigma_1(n) q^n
\]
is the (holomorphic) Eisenstein series of weight $2$. Unlike $E_{2m}$ for $m\ge 2$,
the function $E_2$ is not a modular form. Instead it satisfies the \emph{quasi-modular}
transformation law
\[
E_2(\gamma z)=(cz+d)^2E_2(z)+\frac{12}{2\pi i}c(cz+d),
\qquad
\gamma=\begin{pmatrix}a&b\\ c&d\end{pmatrix}\in\SL_2(\mathbb{Z}).
\]
More generally, a holomorphic function $f:\Hh\to\mathbb{C}$ is called a
\emph{quasi-modular form} of weight $k$ and depth $\le p$ (for $\SL_2(\mathbb{Z})$)
if for every $\gamma=\begin{pmatrix}a&b\\ c&d\end{pmatrix}\in\SL_2(\mathbb{Z})$ one has
\[
f(\gamma z)=(cz+d)^k\sum_{r=0}^p f_r(z)\Bigl(\frac{c}{cz+d}\Bigr)^r
\]
for some holomorphic functions $f_r$ on $\Hh$. The graded ring of quasi-modular forms
for $\SL_2(\mathbb{Z})$ is generated by $E_2,E_4,E_6$; see \cite{L7}.

The Ramanujan discriminant modular form of weight $12$ is
\[
\Delta(z) = q\prod_{n\ge 1}(1-q^{n})^{24},\qquad
\Delta = \frac{E_4^{3}-E_6^{2}}{1728}.
\]
From the transformation law
\[
\Delta(\gamma z) = (cz+d)^{12}\Delta(z),
\]
one computes
\[
\frac{\Delta'}{\Delta}(\gamma z) = (cz+d)^{2}\frac{\Delta'}{\Delta}(z) + \frac{12}{2\pi i}c(cz+d),
\]
where here $'$ denotes $D$. Hence the function
\[
a(z) = \frac{1}{6}\frac{\Delta'(z)}{\Delta(z)}I
\]
satisfies
\[
a(\gamma z) = (cz+d)^{2}a(z) + \frac{2}{2\pi i}c(cz+d)I,
\]
so $a$ is a classical modular connection of eccentricity $e=1$. By the modular covariant derivative definition, the associated covariant derivative on $M_k(\Gamma,V)$ is
\[
\nabla_a(\psi) = \psi' - \frac{k}{2}a\psi.
\]
Substituting $a=\frac{1}{6}\frac{\Delta'}{\Delta}=\frac{1}{6}E_2I$ yields
\[
\nabla_a(\psi) = \psi' - \frac{k}{12}E_2\psi,
\]
which is precisely the Serre derivative introduced in \cite{L6}.

\subsection{Ramanujan identities }

For $\Gamma=\SL_2(\mathbb{Z})$, the Eisenstein series $E_2,E_4,E_6$ satisfy the Ramanujan differential identities (see Zagier \cite{L7}):
\[
E'_2 = \frac{E_2^{2}-E_4}{12},\qquad
E'_4 = \frac{E_2E_4-E_6}{3},\qquad
E'_6 = \frac{E_2E_6-E_4^{2}}{2}.
\]

\begin{remark}
On $X=\Gamma\backslash \Hh$, let $A$ be a modular connection on $X$ with values in $\End(V)$ and eccentricity $e\neq 0$. Then the curvature
\[
F_A = A' - \frac{1}{2e}A^{2}
\]
defines a well-defined modular $2$-differential on $X$. Indeed, the anomaly term involves the Schwarzian $S(\lambda)$ of the transition function, and $S(\lambda)=0$ for M\"obius changes of coordinates. If $A$ is a classical modular connection with values in $\End(V)$, then $F_A\in M_4(\Gamma,\End(V))$.
\end{remark}

\subsection{Modular Curvature}

\begin{example}
Let $f:\Hh\to V$ be an element of $M_{2m}(\Gamma,V)$, where $V$ is an $n$-dimensional complex vector space. Define the Wronskian
\[
\Wr(f)(z) = f(z)\wedge f'(z)\wedge \cdots \wedge f^{(n-1)}(z)\in \bigwedge^{n}V,
\]
where derivatives are taken with respect to the normalized derivation $D$.
As in Section~2, $\Wr(f)$ transforms as a meromorphic modular form of weight
\[
N := 2mn + n(n-1),
\]
so $\Wr(f)\in M_N(\Gamma,\bigwedge^{n}V)$. Since $\bigwedge^{n}V$ is one-dimensional, we may regard $\Wr(f)$ as a scalar modular form up to a choice of basis. A change of basis multiplies $\Wr(f)$ by $\det(T)$, hence the logarithmic derivative $\Wr(f)'/\Wr(f)$ is independent of the choice of basis.

Define the scalar Maurer--Cartan connection
\[
A = \frac{2}{N}\frac{\Wr(f)'}{\Wr(f)}.
\]
Then $A$ is a classical modular connection of eccentricity $e=1$. Consequently its curvature
\[
F_A = A' - \frac{1}{2}A^{2}
\]
is a meromorphic modular form of weight $4$. Thus we obtain a natural nonlinear operator
\[
M_{2m}(\Gamma,V)\longrightarrow M_4(\Gamma,\mathbb{C}),\qquad f\longmapsto F_A.
\]
\end{example}

\subsection{The Chazy equation}

In the classical Cartan theory of Lie-algebra-valued connections, the curvature
$F=dA + \frac{1}{2}[A,A]$ satisfies the Bianchi identity $d_A F=0$. In our one-dimensional framework, the curvature of a connection is the quadratic expression
\[
F_A = A' - \frac{1}{2e}A^{2},
\]
which should be viewed as a curvature-type $2$-differential. However, as we saw in Section~3, $F_A$ is not globally tensorial on an arbitrary Riemann surface: under changes of coordinate it acquires an anomaly term proportional to the Schwarzian derivative of the transition map. Thus one should not expect a literal analogue of the classical Bianchi identity in general.
Chazy introduced in \cite{L8} a class of nonlinear third-order differential equations whose general solutions have fixed critical points. A distinguished example is the Chazy equation
\begin{equation}\label{eq:4.2}
2u''' - 2uu'' + 3(u')^{2} = 0,
\end{equation}
where, in this subsection, primes denote the normalized derivation $D$.
It is a classical fact (see Zagier \cite{L7}) that the Eisenstein series $E_2$ satisfies \eqref{eq:4.2}.
We now derive this from the modular connection $A=\frac{1}{6}\frac{\Delta'}{\Delta}$.

\begin{lemma}[ Chazy equation]
Let $\Gamma=\SL_2(\mathbb{Z})$ and write $f':=Df$ for $D=\frac{1}{2\pi i}\frac{d}{dz}$. Let
\[
A=\frac{1}{6}\frac{\Delta'}{\Delta}=\frac{1}{6}E_2,
\]
so that $A$ is a scalar classical modular connection of eccentricity $e=1$.
Let $F_A = A' - \frac{1}{2}A^{2}$ be the curvature. Then $A$ satisfies the Chazy equation
\[
\nabla_A^{2}F_A + 12F_A^{2} = 0,
\]
\end{lemma}

\part{Periods}\label{sec:periods}

In this section we apply the ideas developed so far to study periods of holomorphic
$1$-forms on a one-parameter family of compact Riemann surfaces of fixed genus $g$,
generalizing the formula of Dedekind mentioned in the introduction. Differential equations satisfied by periods of holomorphic $1$-forms
often fall under the title of \emph{Picard--Fuchs equations}: these are
scalar linear ODEs of order $2g$ obtained from the Gauss--Manin system.
While extremely useful, such scalar equations are not well suited for the
Schwarzian point of view pursued in this paper. For $g>1$ we therefore
propose a different approach which produces a canonical second order ODE
with $g\times g$ matrix coefficients in $\End(E)$, and from there we obtain
new insight about periods using the matrix Schwarzian derivative.
Such a
one-parameter family, roughly speaking, either corresponds to a non-constant
holomorphic map from a Riemann surface $S$ to the moduli space of curves of genus
$g$, $\mathcal{M}_g$, or equivalently to a holomorphic fibration
$\pi\colon X\to S$ from a complex surface $X$ to $S$ whose fibers
$X_s=\pi^{-1}(s)$ are Riemann surfaces of genus $g$.

For a compact Riemann surface $C$ of genus $g$, its first cohomology group
$H^1(C,\mathbb{C})$ carries a natural Hodge decomposition
\[
H^1(C,\mathbb{C}) \;=\; H^{1,0}(C) \oplus H^{0,1}(C),
\]
where $H^{1,0}(C)\simeq H^0(C,K_C)$ is the space of holomorphic $1$-forms and
$H^{0,1}(C)\simeq \overline{H^{1,0}(C)}$ is the space of anti-holomorphic
$1$-forms.

Associated to a ``generic'' family $\pi\colon X\to S$ (to be made precise
later), or equivalently to the induced moduli map
\[
\phi\colon S\to \mathcal M_g,\qquad s\longmapsto [X_s],
\]
we will construct a classical connection
\[
A=A_\phi
\]
of eccentricity $e=\tfrac12$ with values in $\End(E)$, where
\[
E:=\pi_*\Omega^1_{X/S}
\]
is the Hodge bundle whose fiber at $s\in S$ is $H^{1,0}(X_s)$.
We will also construct a classical $2$-differential
\[
q=q_\phi
\]
with values in $\End(E)$.
The associated second order ODE with matrix coefficients
\[
\Psi'' = 2A\Psi' + q\Psi
\]
will then encode the variation of periods: for a locally constant homology
class $\gamma\in H_1(X_s,\mathbb{Z})$ and a holomorphic $1$-form
$\omega_s\in H^{1,0}(X_s)$ varying holomorphically with $s$, the period function
\[
s\longmapsto \int_{\gamma_s}\omega_s
\]
arises as a component of a solution $\Psi$ of the above equation.

We briefly recall that the traditional Picard--Fuchs approach starts from the
Gauss--Manin system (a first order linear system of rank $2g$) and produces a
scalar linear ODE of order $2g$ by choosing a cyclic vector in the sense of Katz.
However, for $g>1$ there is in general no canonical choice of such a cyclic vector
(or of a distinguished scalar period), so the resulting scalar equation depends on
auxiliary choices. By contrast, the second order ODE with coefficients in $\End(E)$
that we construct below is canonical: it is determined intrinsically by the family
(under an explicit genericity hypothesis) and is compatible with the gauge--theoretic
formalism developed in the previous sections.

\subsection{Dedekind's formula}\label{subsec:dedekind}

We illustrate the preceding discussion in the case of elliptic curves ($g=1$),
recovering Dedekind's Schwarzian differential equation from a Picard--Fuchs
computation. Throughout this subsection, $j$ denotes the normalized Hauptmodul
on $X(1)\simeq \Pp^1$ used in the introduction, so that the three special points
correspond to $j=0,1,\infty$.

\medskip
\noindent\textbf{The universal family.}
Consider the one-parameter family of smooth plane cubic curves
\begin{equation}\label{eq:dedekind-family}
E_j:\qquad y^2 = 4x^3 - g(j)\,x - g(j),
\qquad g(j)=\frac{27j}{j-1},
\end{equation}
over $U:=\Pp^1\setminus\{0,1,\infty\}$.
For the Weierstrass equation $y^2=4x^3-g_2x-g_3$, one has discriminant
$\Delta=g_2^3-27g_3^2$ and normalized modular invariant
\[
j=\frac{g_2^3}{\Delta}.
\]
In our family $g_2=g_3=g(j)$, hence
\[
\Delta = g(j)^2\bigl(g(j)-27\bigr),\qquad
\frac{g(j)^3}{\Delta}=\frac{g(j)}{g(j)-27}=j,
\]
so~\eqref{eq:dedekind-family} indeed gives the universal elliptic curve over the
$j$-line in this normalization (cf.~\cite{L4}).

\medskip
\noindent\textbf{Gauss--Manin system.}
Let
\[
\omega=\frac{dx}{y},\qquad \eta=\frac{x\,dx}{y}
\]
be the standard holomorphic differential and a differential of the second kind on
$E_j$. Then $\{\omega,\eta\}$ is a basis of $H^1_{\mathrm{dR}}(E_j)$.
Differentiating in the parameter $g$ (and using $y^2=4x^3-gx-g$) gives
\[
\partial_g\omega = \frac{x+1}{2}\frac{dx}{y^3},\qquad
\partial_g\eta = \frac{x(x+1)}{2}\frac{dx}{y^3}.
\]
By a standard reduction step (subtracting exact differentials of the form
$d(H/y)$; see \cite{L9}), one can express these derivatives in the basis
$\{\omega,\eta\}$.

\begin{lemma}\label{lem:dedekind-gm}
In $H^1_{\mathrm{dR}}(E_j)$ one has
\begin{equation}\label{eq:dedekind-gm}
\frac{d}{dg}
\binom{\omega}{\eta}
=
\begin{pmatrix}
\displaystyle \frac{18-g}{4g(g-27)} & \displaystyle -\frac{3}{2g(g-27)}\\[8pt]
\displaystyle \frac{1}{8(g-27)} & \displaystyle \frac{g-18}{4g(g-27)}
\end{pmatrix}
\binom{\omega}{\eta},
\end{equation}
where $g=g(j)$ as in~\eqref{eq:dedekind-family}.
\end{lemma}

\begin{proof}
We sketch the reduction for $\partial_g\omega$; the computation for
$\partial_g\eta$ is similar.
Using $d(H/y)=H'\,dx/y - \frac12H\,(12x^2-g)\,dx/y^3$ and the identity
$dx/y = (4x^3-gx-g)\,dx/y^3$, a direct calculation shows that
\[
\partial_g\omega
=
\frac{18-g}{4g(g-27)}\,\omega
-\frac{3}{2g(g-27)}\,\eta
+
d\!\left(\frac{H}{y}\right),
\]
where one may take
\[
H(x)=\frac{1}{g-27}\Bigl(\frac{3}{g}x^2+\frac{18-g}{2g}x-\frac12\Bigr).
\]

Similarly, one checks that
\[
\partial_g\eta
=
\frac{1}{8(g-27)}\,\omega
+\frac{g-18}{4g(g-27)}\,\eta
+
d\!\left(\frac{H_2}{y}\right),
\]
where
\[
H_2(x)=\frac{1}{g-27}\Bigl(\frac{18-g}{2g}x^2+\frac14x+\frac34\Bigr).
\]
which yields the second row of~\eqref{eq:dedekind-gm}.

Since exact forms integrate to zero over closed cycles, this yields the first row
of~\eqref{eq:dedekind-gm}. The second row follows from an analogous reduction of
$\partial_g\eta$.
\end{proof}

Let $\gamma$ be a locally constant cycle in $H_1(E_j,\mathbb{Z})$.
Integrating~\eqref{eq:dedekind-gm} over $\gamma$ yields a first order linear system
for the period vector
\[
\Pi_\gamma(g)=\binom{\int_\gamma\omega}{\int_\gamma\eta}.
\]
Eliminating the second component gives a scalar second-order equation for
$y(g):=\int_\gamma\omega$.

\begin{proposition}[Picard--Fuchs equation in the $j$--coordinate]\label{prop:dedekind-pf}
The periods $y(j)=\int_\gamma\omega$ satisfy the Picard--Fuchs equation
\begin{equation}\label{eq:dedekind-pf}
\frac{d^2y}{dj^2} + \frac{1}{j}\frac{dy}{dj}
+ \frac{31j-4}{144j^2(1-j)^2}\,y = 0.
\end{equation}
\end{proposition}

\begin{proof}
Let $\gamma$ be a locally constant cycle and set
\[
y(g):=\int_{\gamma}\omega,\qquad z(g):=\int_{\gamma}\eta.
\]
Integrating~\eqref{eq:dedekind-gm} gives a first-order system
\[
\frac{d}{dg}\binom{y}{z}
=
\begin{pmatrix}
a & b\\
c & -a
\end{pmatrix}
\binom{y}{z},
\qquad
\begin{aligned}
a&=\frac{18-g}{4g(g-27)},\\
b&=-\frac{3}{2g(g-27)},\\
c&=\frac{1}{8(g-27)}.
\end{aligned}
\]
From the first equation we have $z=(y'-ay)/b$. Differentiating this identity and
using the second equation to eliminate $z'$ yields the scalar equation
\[
\frac{d^2y}{dg^2}+\frac{2g-27}{g(g-27)}\frac{dy}{dg}
+\frac{3(g+4)}{16g^2(g-27)}\,y=0.
\]
Finally, using the change of variable $g=\frac{27j}{j-1}$ and the chain rule converts
this into~\eqref{eq:dedekind-pf}. This is the equation stated in the introduction.
\end{proof}

\medskip
\noindent\textbf{Dedekind's Schwarzian equation.}
Let $\omega_1(j),\omega_2(j)$ be two linearly independent solutions of
\eqref{eq:dedekind-pf} and set $\tau(j)=\omega_2(j)/\omega_1(j)$.
Changing the ordered basis $(\omega_1,\omega_2)$ acts on $\tau$ by a M\"obius
transformation, hence $S(\tau)$ is well defined. For a choice of period basis
the ambiguity is the natural action of $\SL_2(\mathbb{Z})$ on $\tau$.
In particular, $\tau(j)$ depends only on the period lattice and gives a
uniformizing parameter on $\Hh$.
Putting~\eqref{eq:dedekind-pf} into normal form by removing the first derivative term,
one finds that $\tau$ satisfies a Schwarzian differential equation.
More precisely, from \eqref{eq:dedekind-pf} one computes
\[
S(\tau)(j)=
2\left(b-\frac12a'-\frac14a^2\right),
\qquad
a=\frac{1}{j},\quad b=\frac{31j-4}{144j^2(1-j)^2},
\]
and hence
\begin{equation}\label{eq:dedekind-schwarzian}
S(\tau)(j) = \frac{3}{8(1-j)^2} + \frac{4}{9j^2} + \frac{23}{72j(1-j)}.
\end{equation}
This is precisely the formula found by Dedekind \cite{L4}; compare the discussion in
the introduction. In the language of Section~4, rewriting~\eqref{eq:dedekind-pf} in the form
$y''=2Ay'+qy$ with
\[
A=-\frac{1}{2j},\qquad q=-\frac{31j-4}{144j^2(1-j)^2},
\]
one obtains $2(F_A-q)=S(\tau)(j)$.
We now proceed to extend this construction to higher genus $g>1$, where the ODE remains of second order but the coefficients become $g\times g$ matrices.

\section{Higher genus periods}\label{sec:higher-genus-periods}

On the moduli space $\mathcal{M}_g$ there is a canonical holomorphic vector
bundle $\mathbb{E}$ of rank $g$ (the \emph{Hodge bundle}) whose fiber at a point
$[C]\in\mathcal{M}_g$ is
\[
\mathbb{E}_{[C]}\;\cong\;H^{1,0}(C)\;\cong\;H^0(C,K_C).
\]
In algebro--geometric terms, if $\pi\colon \mathcal{C}_g\to\mathcal{M}_g$ denotes
the universal curve (viewed as a stack), then
\[
\mathbb{E}=\pi_*\omega_{\pi},
\]
where $\omega_{\pi}$ is the relative dualizing sheaf \cite{L11}.
For a family $\pi\colon X\to S$ with induced moduli map $\phi\colon S\to\mathcal{M}_g$,
the bundle $E=\pi_*\Omega^1_{X/S}$ introduced above is naturally identified with the
pullback $\phi^*\mathbb{E}$.

Alongside $\mathbb{E}$ one has the rank $2g$ local system
\[
\mathbb{H}:=R^1\pi_*\mathbb{C}
\]
on $S$, whose fiber at $s\in S$ is $H^1(X_s,\mathbb{C})$.
The associated holomorphic vector bundle
\[
\mathcal{H}:=\mathbb{H}\otimes_{\mathbb{C}}\mathcal{O}_S\cong H^1_{\mathrm{dR}}(X/S)
\]
admits the \emph{Gauss--Manin connection}
\[
\nabla^{\mathrm{GM}}\colon \mathcal{H}\longrightarrow \mathcal{H}\otimes\Omega^1_S,
\qquad (\nabla^{\mathrm{GM}})^2=0,
\]
which is a canonical integrable (flat) connection encoding the parallel
transport of cohomology classes in families \cite{L9}.
The Hodge decomposition of each fiber $H^1(X_s,\mathbb{C})\cong H^{1,0}(X_s)\oplus H^{0,1}(X_s)$
is not preserved by $\nabla^{\mathrm{GM}}$; instead one has Griffiths transversality
\(
\nabla^{\mathrm{GM}}(E)\subset (\mathcal{H}/E)\otimes\Omega^1_S
\), which reflects the variation of complex structure.

Fix a simply connected open set $U\subset S$ with coordinate $t$.
Choosing a flat basis of cycles $\gamma_1,\dots,\gamma_{2g}$ for the local system
$R_1\pi_*\mathbb{Z}$ and a holomorphic frame $\omega_1,\dots,\omega_g$ of $E|_U$,
the period matrix
\[
\Pi(t):=\bigl(\,\int_{\gamma_j(t)}\omega_i(t)\,\bigr)_{1\le i\le g,\,1\le j\le 2g}
\]
satisfies a first order linear differential system (the \emph{Gauss--Manin system})
obtained by expressing $\nabla^{\mathrm{GM}}\omega_i$ in the chosen basis.
Equivalently, the vector of all period functions is a horizontal section of
$(\mathcal{H},\nabla^{\mathrm{GM}})$.

In the one-parameter case one may also extract from the rank $2g$ Gauss--Manin system
a scalar linear ordinary differential equation of order $2g$ (a Picard--Fuchs equation)
annihilating a given period function.
Katz gives an algebraic construction of the Gauss--Manin connection and a method to pass
from the first-order system to such a scalar equation by choosing a cyclic vector
\cite{L10,L12}. This reduction is highly useful, but it is not canonical in general:
the resulting scalar equation depends on the choice of cyclic vector (equivalently, on the
elimination procedure).
In this section we instead emphasize the canonical second order ODE with coefficients in
$\End(E)$ attached to a generic family.

\begin{definition}\label{def:generic-family}
Let $\pi\colon X\to S$ be a holomorphic family of compact Riemann surfaces of
genus $g\ge 1$, and let $E=\pi_*\Omega^1_{X/S}$ be the Hodge bundle on $S$.
On a simply connected open set $U\subset S$ with local coordinate $s$, let
\[
\theta_{\partial_s}\colon E|_U \longrightarrow (R^1\pi_*\mathcal{O}_X)|_U
\]
denote the Higgs field (equivalently, the $(0,1)$-part of the Gauss--Manin
connection along $\partial_s$).
We say that the family $\pi$ is \emph{generic on $U$} if $\theta_{\partial_s}$
is an isomorphism on a dense open subset of $U$ (equivalently, if
$\det(\theta_{\partial_s})$ is not identically zero on $U$).
We say that $\pi$ is \emph{generic} if it is generic on every sufficiently
small coordinate chart $U\subset S$.
\end{definition}

\begin{remark}\label{rem:generic-family-equivalences}
Since $\operatorname{rank}(E)=\operatorname{rank}(R^1\pi_*\mathcal{O}_X)=g$, the
condition in Definition~\ref{def:generic-family} is equivalent to requiring
that for a generic point $s\in U$ one can choose a holomorphic frame
$\omega_1,\dots,\omega_g$ of $E|_U$ such that the classes of
$\partial_s\omega_1,\dots,\partial_s\omega_g$ span $H^{0,1}(X_s)$.
Equivalently, for a generic $s\in U$ the induced bilinear form
\[
(\omega,\eta)\longmapsto \langle \kappa_s(\partial_s),\,\omega\eta\rangle
\qquad \bigl(\omega,\eta\in H^{1,0}(X_s)\bigr)
\]
is nondegenerate, where $\kappa_s(\partial_s)\in H^1(X_s,T_{X_s})$ is the
Kodaira--Spencer class.
\end{remark}

\begin{definition}\label{def:period-Aq}
Let $\pi\colon X\to S$ be a generic family of compact Riemann surfaces of genus $g$
over a Riemann surface $S$, and let
\[
E:=\pi_*\Omega^1_{X/S}
\]
be the Hodge bundle on $S$. Let $\mathcal{H}=H^1_{\mathrm{dR}}(X/S)$ denote the
relative de Rham bundle equipped with the Gauss--Manin connection
$\nabla^{\mathrm{GM}}$.
Fix a simply connected coordinate chart $U\subset S$ with local coordinate $s$
on which the genericity condition holds.
Choosing a flat trivialization of $(\mathcal{H},\nabla^{\mathrm{GM}})$ over $U$,
we identify covariant differentiation $\nabla^{\mathrm{GM}}_{\partial_s}$ with
ordinary differentiation and write primes for $d/ds$.

Assume that $E|_{U}$ admits a holomorphic frame $\omega_1,\dots,\omega_g$ such
that the $2g$ sections
\[
\omega_1,\dots,\omega_g,\ \omega_1':=\partial_s\omega_1,\dots,\partial_s\omega_g=:\omega_g'
\]
form a frame of $\mathcal{H}|_{U}$ (equivalently, the matrix of the Higgs field
$\theta_{\partial_s}$ in the frame $\omega_1,\dots,\omega_g$ is pointwise
invertible on a dense open subset of $U$).
Let
\[
\omega:=\begin{pmatrix}\omega_1\\ \vdots\\ \omega_g\end{pmatrix}.
\]
Since $(\omega,\omega')$ is a frame of $\mathcal{H}|_{U}$, there exist unique
meromorphic matrix-valued functions
\[
A_U,\ q_U\in \M(U,\End(\mathbb{C}^g))
\]
such that
\[
\omega'' \;=\; 2A_U\,\omega' \;+\; q_U\,\omega.
\]
We call $A_U$ and $q_U$ the \emph{local coefficient matrices} attached to the
family $\pi$ (or equivalently to the moduli map $\phi$) on $U$.
\end{definition}

\begin{theorem}\label{thm:period-Aq-transformation}
In the setting of Definition~\ref{def:period-Aq}, the collections
$\{A_U\}$ and $\{q_U\}$ glue to define a classical connection $A=A_\phi$ on $E$
of eccentricity $e=1/2$ and a classical $2$-differential $q=q_\phi$ with
values in $\End(E)$.

More precisely, let $(U_i,s_i,\omega^{(i)})$ and $(U_j,s_j,\omega^{(j)})$ be two
choices as in Definition~\ref{def:period-Aq}, and write
$\lambda_{ij}=s_i\circ s_j^{-1}$ for the change of coordinate on $U_i\cap U_j$.
Assume that the corresponding frames are related by a constant matrix
$g_{ij}\in \GL_g(\mathbb{C})$,
\[
\omega^{(j)} = g_{ji}\,(\omega^{(i)}\circ \lambda_{ij}),
\qquad g_{ji}=g_{ij}^{-1}.
\]
Then on $U_i\cap U_j$ one has
\[
A_j
=
g_{ji}\Bigl(\lambda'_{ij}(A_i\circ \lambda_{ij})+\frac12\frac{\lambda''_{ij}}{\lambda'_{ij}}I\Bigr)g_{ij},
\qquad
q_j
=
(\lambda'_{ij})^2\,g_{ji}\,(q_i\circ \lambda_{ij})\,g_{ij}.
\]
In particular, $A$ transforms as a (classical) connection of eccentricity
$e$ and $q$ transforms as an $\End(E)$-valued $2$-differential.
\end{theorem}

\begin{proof}
The key point is that the relative de~Rham bundle
\[
\mathcal H = H^1_{\mathrm{dR}}(X/S)
\]
carries the flat Gauss--Manin connection. Hence on each simply connected chart
$U_i$ we may choose a flat trivialization of $\mathcal H|_{U_i}$, and on an
overlap $U_i\cap U_j$ two such trivializations differ by a locally constant
matrix. After shrinking if necessary, we may therefore write the two Hodge
frames in the form
\[
\omega^{(j)} = g_{ji}\,(\omega^{(i)}\circ \lambda_{ij}),
\qquad g_{ji}=g_{ij}^{-1},
\]
with $g_{ij}\in \GL_g(\mathbb C)$ locally constant on $U_i\cap U_j$.
This justifies the constant transition matrix appearing in the statement.

We next verify that the second-order system is canonical. On a chart
$(U_i,s_i)$, genericity implies that
\[
\omega^{(i)}_1,\dots,\omega^{(i)}_g,\;
\partial_{s_i}\omega^{(i)}_1,\dots,\partial_{s_i}\omega^{(i)}_g
\]
form a frame of $\mathcal H|_{U_i}$ on a dense open subset. Therefore
$(\omega^{(i)},(\omega^{(i)})')$ is a basis there, and so the second derivative
$(\omega^{(i)})''$ can be written uniquely as
\[
(\omega^{(i)})'' = 2A_i\,(\omega^{(i)})' + q_i\,\omega^{(i)}.
\]
Thus the matrices $A_i$ and $q_i$ are uniquely determined by the family and the
chosen coordinate, so the resulting second-order system is canonical once the
Hodge frame is fixed.

Now fix a point of $U_i\cap U_j$ and write $\lambda=\lambda_{ij}=s_i\circ s_j^{-1}$.
Set $p_i:=2A_i$. Consider first the same vector-valued section written in the
$s_j$-coordinate but before changing the flat frame:
\[
\widetilde\omega:=\omega^{(i)}\circ\lambda.
\]
By Lemma~\ref{lem:quantum-second-order-coordinates}, applied componentwise to
the vector-valued function $\widetilde\omega$, the equation
\[
(\omega^{(i)})''=p_i\,(\omega^{(i)})'+q_i\,\omega^{(i)}
\]
transforms into
\[
\widetilde\omega''=\widetilde p_j\,\widetilde\omega'+\widetilde q_j\,\widetilde\omega
\]
with
\[
\widetilde p_j=\lambda'(p_i\circ\lambda)+\frac{\lambda''}{\lambda'}I,
\qquad
\widetilde q_j=(\lambda')^2(q_i\circ\lambda).
\]
Writing $\widetilde p_j=2\widetilde A_j$, we get
\[
\widetilde A_j=\lambda'(A_i\circ\lambda)+\frac12\frac{\lambda''}{\lambda'}I.
\]

Finally, passing from $\widetilde\omega$ to the chosen Hodge frame
\[
\omega^{(j)}=g_{ji}\,\widetilde\omega,
\]
with $g_{ji}$ locally constant, differentiation commutes with multiplication by
$g_{ji}$. Hence $\omega^{(j)}$ satisfies
\[
(\omega^{(j)})'' = 2A_j(\omega^{(j)})' + q_j\omega^{(j)}
\]
where
\[
A_j=g_{ji}\,\widetilde A_j\,g_{ij},
\qquad
q_j=g_{ji}\,\widetilde q_j\,g_{ij}.
\]
Substituting the formulas for $\widetilde A_j$ and $\widetilde q_j$ gives
\[
A_j
=
g_{ji}\Bigl(\lambda'(A_i\circ\lambda)+\frac12\frac{\lambda''}{\lambda'}I\Bigr)g_{ij},
\qquad
q_j
=
(\lambda')^2\,g_{ji}\,(q_i\circ\lambda)\,g_{ij},
\]
which is exactly the claimed transformation law. Therefore the local matrices
$\{A_U\}$ glue to a classical connection of eccentricity $e=\tfrac12$ on $E$,
and the local matrices $\{q_U\}$ glue to an $\End(E)$-valued classical
$2$-differential.
\end{proof}

\section{A genus two family }\label{sec:genus-two-family}

We conclude the section with a concrete genus two example.
Recall that every smooth curve of genus $2$ is hyperelliptic, so hyperelliptic
families provide a natural and sufficiently general testing ground for the
matrix Schwarzian formalism.

Let $U\subset\mathbb{C}$ be a simply connected open set avoiding the discriminant locus,
and consider the one--parameter family of genus $2$ curves
\begin{equation}\label{eq:genus2-family}
X_t:\qquad y^2=x^5+t\,x^2+1,\qquad t\in U.
\end{equation}
The fiber $X_t$ is smooth if and only if
\[
D(t):=\operatorname{Res}_x(x^5+t\,x^2+1,\ 5x^4+2tx)=108t^5+3125
\]
is nonzero.

On $X_t$ the holomorphic $1$-forms are spanned by
\[
\omega_1=\frac{dx}{y},\qquad \omega_2=\frac{x\,dx}{y}.
\]
Writing $\omega=(\omega_1,\omega_2)^t$, differentiation with respect to $t$
gives meromorphic differentials of the second kind
\[
\omega_1'=-\frac12\,\frac{x^2\,dx}{y^3},\qquad
\omega_2'=-\frac12\,\frac{x^3\,dx}{y^3}.
\]
A direct computation in $H^1_{\mathrm{dR}}(X_t)$ (for instance via Griffiths
reduction) shows that the $4$ classes
\[
\omega_1,\ \omega_2,\ \omega_1',\ \omega_2'
\]
form a basis for $H^1_{\mathrm{dR}}(X_t)$ for $D(t)\neq 0$, so the family is
generic in the sense of Definition~\ref{def:generic-family} on $U$.
Therefore, by Definition~\ref{def:period-Aq} there exist unique $2\times 2$
matrix--valued functions $A(t)$ and $q(t)$ such that
\begin{equation}\label{eq:genus2-second-order}
\omega''=2A(t)\,\omega'+q(t)\,\omega.
\end{equation}
For the family~\eqref{eq:genus2-family} one finds explicitly
\begin{equation}\label{eq:genus2-Aq}
A(t)=\frac{1}{D(t)}
\begin{pmatrix}
-108t^4 & -675t^2\\
750t & -162t^4
\end{pmatrix},
\qquad
q(t)=\frac{1}{D(t)}
\begin{pmatrix}
-27t^3 & -\frac{275}{2}t\\[2pt]
\frac{375}{2} & -33t^3
\end{pmatrix}.
\end{equation}

Since $e=\tfrac12$, the curvature of $A$ is $F_A=A'-A^2$ and the (matrix) quantum
Schwarzian associated to~\eqref{eq:genus2-second-order} is
\[
\mathcal S_{A,q}:=2(F_A-q).
\]
Using~\eqref{eq:genus2-Aq} one obtains the following $\End(\mathbb{C}^2)$-valued
$rational$ $2$-differential:
\begin{equation}\label{eq:genus2-matrix-schwarzian}
\mathcal S_{A,q}(t)=\frac{1}{D(t)^2}
\begin{pmatrix}
t^3(5832t^5-1518750) & 25t(4104t^5-303125)\\[3pt]
1125(3125-252t^5) & t^3(-10368t^5-2831250)
\end{pmatrix}.
\end{equation}

\begin{remark}\label{rem:genus2-characteristic}
The matrix Schwarzian $\mathcal S_{A,q}$ produces scalar invariants by taking
characteristic combinations. In the present genus two case one obtains in
particular the $2$-differentials
\[
\operatorname{tr}(\mathcal S_{A,q})\qquad\text{and}\qquad \operatorname{tr}(\mathcal S_{A,q}^2),
\]
which are well defined on each chart and hence on $X$ in the same sense as the
quantum curvature (they inherit the same Schwarzian anomaly under coordinate
change). By Lemma~\ref{lem:gauge-curvature-ehalf} and
Remark~\ref{rem:quantum-characteristic-invariants}, these are the first
instances of the \emph{quantum characteristic invariants} attached to the
second-order equation~\eqref{eq:genus2-second-order}.

In genus one, Dedekind's computation expresses the Schwarzian derivative of the
ratio of two periods as $2(F_A-q)$ for a scalar equation; see
\eqref{eq:dedekind-schwarzian}. Lemma~\ref{lem:gauge-curvature-ehalf} shows that
$2(F_A-q)$ is the intrinsic Schwarzian object associated to a second order
equation, and the present construction may be viewed as its extension from
elliptic curves to families of curves of arbitrary genus $g$, where the ODE
remains of second order but the coefficients become $g\times g$ matrices and
the resulting matrix Schwarzian yields canonical characteristic invariants.
Finally, since every genus two curve is hyperelliptic, the family
\eqref{eq:genus2-family} is a natural test case for the higher genus theory.
\end{remark}

\section{Cubic threefolds in $\mathbb{P}^4$}\label{sec:cubic-threefolds}

In this section we illustrate how the second order non-abelian formalism developed in
Section~\ref{sec:classical-quantum-second-order-ode} and Part~\ref{sec:periods} extends to a
basic higher-dimensional moduli problem, namely smooth cubic threefolds in $\mathbb{P}^4$.
While the Gauss--Manin connection for a one-parameter family of cubic threefolds is a
first-order system of rank $10$ on $H^3$, our approach produces a canonical \emph{second
order} system on the rank-$5$ Hodge bundle $E=H^{2,1}$ with coefficients in $\End(E)$.
For a natural deformation of the Fermat cubic, a large symmetry group forces this system
to split into one-dimensional character spaces; in a symmetry-adapted basis the
coefficient matrices become diagonal. We explain the geometric input (Hodge theory,
residues, and the Gauss--Manin system), state the resulting diagonal second order ODE,
and omit only the standard reduction computations.

\subsection{Cubic threefolds}

A \emph{cubic threefold} is a smooth hypersurface $X\subset\mathbb{P}^4$ defined by a
homogeneous polynomial $F(x_0,\dots,x_4)=0$ of degree $3$. Two classical examples are
the \emph{Fermat cubic threefold}
\[
X_{\mathrm{F}}:\quad x_0^3+x_1^3+x_2^3+x_3^3+x_4^3=0,
\]
and the \emph{Klein cubic threefold}
\[
X_{\mathrm{K}}:\quad x_0^2x_1+x_1^2x_2+x_2^2x_3+x_3^2x_4+x_4^2x_0=0.
\]
Both are smooth and admit large automorphism groups.

The Hodge numbers of a smooth cubic threefold are classical:
\[
h^{3,0}(X)=h^{0,3}(X)=0,\qquad h^{2,1}(X)=h^{1,2}(X)=5,\qquad h^{1,1}(X)=1,
\]
so that
\[
H^3(X,\mathbb{C})=H^{2,1}(X)\oplus H^{1,2}(X),\qquad \dim H^{2,1}(X)=5.
\]
Equivalently, the Hodge diamond has the form shown in Figure~\ref{fig:cubic-threefold-hodge-diamond}; see for instance \cite{L18,L20}.

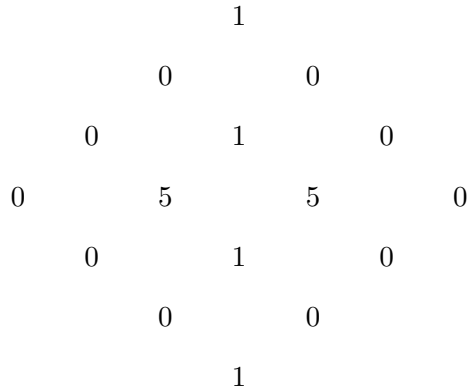
\begin{figure}[H]
\centering
\begin{tikzpicture}[scale=1.0, transform shape, x=1.4cm,y=0.8cm, every node/.style={inner sep=1pt}]
\node at (0,3) {$1$};
\node at (-0.7,2) {$0$};
\node at (0.7,2) {$0$};
\node at (-1.4,1) {$0$};
\node at (0,1) {$1$};
\node at (1.4,1) {$0$};
\node at (-2.1,0) {$0$};
\node at (-0.7,0) {$5$};
\node at (0.7,0) {$5$};
\node at (2.1,0) {$0$};
\node at (-1.4,-1) {$0$};
\node at (0,-1) {$1$};
\node at (1.4,-1) {$0$};
\node at (-0.7,-2) {$0$};
\node at (0.7,-2) {$0$};
\node at (0,-3) {$1$};
\end{tikzpicture}
\caption{Hodge diamond of a smooth cubic threefold.}
\label{fig:cubic-threefold-hodge-diamond}
\end{figure}

Since $h^{3,0}=0$, the Hodge structure on $H^3(X)$ is of type $(2,1)+(1,2)$. The associated
\emph{intermediate Jacobian} can be presented as the complex torus
\[
J(X):=H^{1,2}(X)\big/ H^3(X,\mathbb{Z}),
\]
where we view $H^3(X,\mathbb{Z})$ as a lattice in $H^{1,2}(X)$ via the projection
$H^3(X,\mathbb{C})\twoheadrightarrow H^{1,2}(X)$ coming from the Hodge decomposition
$H^3(X,\mathbb{C})=H^{2,1}(X)\oplus H^{1,2}(X)$. It is a principally polarized abelian variety of
dimension $5$ \cite{L18}. The intermediate Jacobian plays a central role in the
geometry of cubic threefolds (for instance in the Clemens--Griffiths irrationality
theorem).

\subsection{Fermat cubic}

This family is chosen because it gives a fully explicit computable model of the canonical second-order system.

Let
\[
X_0:\quad x_0^3+x_1^3+x_2^3+x_3^3+x_4^3=0
\]
be the Fermat cubic threefold. We consider the one-parameter deformation
\begin{equation}\label{eq:cubic-family}
X_t:\quad F_t:=x_0^3+x_1^3+x_2^3+x_3^3+x_4^3-3t\,x_0x_1x_2=0,
\qquad t\in U\subset\mathbb{C},
\end{equation}
which is smooth for generic $t$ (i.e.\ away from the discriminant locus).
Write $\pi:\mathcal{X}\to U$ for the corresponding family.

Let $\Omega$ be the standard homogeneous $4$-form on $\mathbb{C}^5$,
\[
\Omega=\sum_{i=0}^4 (-1)^i x_i\,dx_0\wedge\cdots\wedge \widehat{dx_i}\wedge\cdots\wedge dx_4.
\]
By Griffiths' description of primitive cohomology of smooth hypersurfaces, the Hodge
bundle $E=\pi_*\Omega^{2}_{\mathcal{X}/U}$ can be realized via residues \cite{L19,L20}:
\[
H^{2,1}(X_t)\ \cong\ \Bigl\{\ \operatorname{Res}_{X_t}\frac{P\Omega}{F_t^2}\ :\
P\in R_1(F_t)\ \Bigr\},
\]
where $R_\bullet(F_t)$ denotes the Jacobian ring of $F_t$. In our cubic threefold
situation one has $R_1(F_t)\cong\mathbb{C}\langle x_0,\dots,x_4\rangle$, hence a
convenient holomorphic frame is given by the residue classes
\begin{equation}\label{eq:residue-basis}
\omega_i(t):=\operatorname{Res}_{X_t}\frac{x_i\,\Omega}{F_t^2},
\qquad i=0,1,2,3,4.
\end{equation}
We write $\omega=(\omega_0,\dots,\omega_4)^t$.

\subsection{Canonical second-order equation }

The Gauss--Manin connection on $H^3_{\mathrm{dR}}(\mathcal{X}/U)$ yields a first order
system of rank $10$. As in Part~\ref{sec:periods}, under a genericity hypothesis
(Definition~\ref{def:generic-family}) one can pass canonically to a second order
equation on $E$ with coefficients in $\End(E)$:
\begin{equation}\label{eq:cubic-second-order}
\omega'' = 2A(t)\,\omega' + q(t)\,\omega,
\qquad A(t),q(t)\in \End(E_t),
\end{equation}
where primes denote differentiation with respect to $t$.
Concretely, differentiating~\eqref{eq:residue-basis} produces residue classes with
higher pole order; Griffiths--Dwork reduction expresses these classes back in the
$\mathbb{C}$-span of $\{\omega_0,\dots,\omega_4,\omega_0',\dots,\omega_4'\}$.
This reduction procedure is standard in the theory of Picard--Fuchs equations for
hypersurfaces; see \cite{L21,L22}. We omit the reduction computations and focus on the
structural features of~\eqref{eq:cubic-second-order} relevant to our non-abelian
Schwarzian viewpoint.

The family~\eqref{eq:cubic-family} is invariant under a finite diagonal symmetry group
$G$ generated by
\[
(x_0,x_1,x_2,x_3,x_4)\longmapsto
(\zeta_0x_0,\zeta_1x_1,\zeta_2x_2,\zeta_3x_3,\zeta_4x_4),
\qquad \zeta_i^3=1,\ \zeta_0\zeta_1\zeta_2=1.
\]
This group acts fiberwise on $X_t$ and hence on $H^{2,1}(X_t)$, and it commutes with the
Gauss--Manin connection. In particular, the coefficient matrices $A(t)$ and $q(t)$ in
\eqref{eq:cubic-second-order} commute with the $G$-action.

The residue classes~\eqref{eq:residue-basis} are eigenvectors for the diagonal action.
Indeed, $F_t$ is $G$-invariant, and $\Omega$ scales by $\zeta_0\zeta_1\zeta_2\zeta_3\zeta_4
=\zeta_3\zeta_4$, so $\omega_i$ transforms by a one-dimensional character of $G$.
Consequently, the $G$-representation on $E_t\simeq H^{2,1}(X_t)$ splits as a direct sum
of five one-dimensional character spaces. By Schur's lemma, any endomorphism commuting
with $G$ preserves each character line. Therefore, in a $G$-eigenbasis the matrices
$A(t)$ and $q(t)$ are diagonal:
\begin{equation}\label{eq:cubic-diagonal}
A(t)=\mathrm{diag}(a_0(t),\dots,a_4(t)),
\qquad
q(t)=\mathrm{diag}(q_0(t),\dots,q_4(t)),
\end{equation}
for meromorphic functions $a_i(t),q_i(t)$ on $U$ (with poles along the discriminant).
Equivalently, the system~\eqref{eq:cubic-second-order} decouples into five scalar second
order ODEs
\[
\omega_i'' = 2a_i(t)\,\omega_i' + q_i(t)\,\omega_i,
\qquad i=0,1,2,3,4.
\]

For the deformation~\eqref{eq:cubic-family}, the diagonal coefficients can be made
completely explicit.

\begin{proposition}\label{prop:cubic-explicit-diagonal}
For the family~\eqref{eq:cubic-family}, in the residue frame~\eqref{eq:residue-basis} the
second-order equation~\eqref{eq:cubic-second-order} has diagonal coefficient matrices
\[
A(t)=\mathrm{diag}(0,0,0,a(t),a(t)),
\qquad
q(t)=\mathrm{diag}(0,0,0,b(t),b(t)),
\]
where
\begin{equation}\label{eq:cubic-explicit-a-b}
a(t)=\frac{3t^2}{2(1-t^3)},
\qquad
b(t)=\frac{t}{1-t^3}.
\end{equation}
Equivalently,
\[
\omega_0'=\omega_1'=\omega_2'=0,
\qquad
\omega_3''=\frac{3t^2}{1-t^3}\,\omega_3' + \frac{t}{1-t^3}\,\omega_3,
\qquad
\omega_4''=\frac{3t^2}{1-t^3}\,\omega_4' + \frac{t}{1-t^3}\,\omega_4.
\]
In the more classical notation
\(\omega''=P(t)\,\omega' + Q(t)\,\omega\)
we have \(P(t)=2A(t)\) and \(Q(t)=q(t)\), hence
\[
P(t)=\mathrm{diag}\Bigl(0,0,0,\frac{3t^2}{1-t^3},\frac{3t^2}{1-t^3}\Bigr),
\qquad
Q(t)=\mathrm{diag}\Bigl(0,0,0,\frac{t}{1-t^3},\frac{t}{1-t^3}\Bigr).
\]
\end{proposition}

\begin{proof}
We sketch the standard reduction argument and omit the routine algebra.
Differentiating~\eqref{eq:residue-basis} with respect to $t$ gives
\[
\omega_i'\;=\;-2\,\operatorname{Res}_{X_t}\frac{x_i\,(\partial_t F_t)\,\Omega}{F_t^3}
\;=\;6\,\operatorname{Res}_{X_t}\frac{x_i\,x_0x_1x_2\,\Omega}{F_t^3}.
\]
For $i=0,1,2$ the corresponding degree--$4$ numerator lies in the Jacobian ideal of
$F_t$, hence its residue class vanishes; this yields $\omega_0'=\omega_1'=\omega_2'=0$.
For $i=3,4$ one differentiates once more, obtaining a residue with pole order $4$ whose
degree--$7$ numerator is a multiple of $(x_0x_1x_2)^2$.
Using the relations
\(\partial_{x_0}F_t=3(x_0^2-tx_1x_2)\),
\(\partial_{x_1}F_t=3(x_1^2-tx_0x_2)\),
\(\partial_{x_2}F_t=3(x_2^2-tx_0x_1)\),
together with the Griffiths--Dwork reduction identity
\cite{L21,L22}, one rewrites the cohomology class represented by this rational form with pole order $4$ in the span of
\(\omega_i\) and \(\omega_i'\). In the present monomial deformation, the reduction is
particularly simple and produces the diagonal coefficients~\eqref{eq:cubic-explicit-a-b}.
\end{proof}

For $e=\tfrac12$ the matrix Schwarzian attached to~\eqref{eq:cubic-second-order} is
\[
\mathcal{S}_{A,q}=2(F_A-q),\qquad F_A=A'-A^2.
\]
In the diagonal frame~\eqref{eq:cubic-diagonal}, $\mathcal{S}_{A,q}$ is diagonal with
entries $2(a_i'-a_i^2-q_i)$; thus the matrix Schwarzian packages scalar Schwarzian
expressions into a single $\End(E)$-valued invariant.

For the explicit monomial deformation~\eqref{eq:cubic-family}, combining
Proposition~\ref{prop:cubic-explicit-diagonal} with the definition of curvature yields a
closed form for the matrix Schwarzian:
\begin{equation}\label{eq:cubic-explicit-schwarzian}
\mathcal{S}_{A,q}(t)
=\mathrm{diag}(0,0,0,s(t),s(t)),
\qquad
s(t)=\frac{t(t^3+8)}{2(1-t^3)^2}.
\end{equation}

More invariantly, by the discussion in
Section~\ref{sec:classical-quantum-second-order-ode} and
Corollary~\ref{cor:charpoly-well-defined}, the conjugacy class of $\mathcal{S}_{A,q}$ is
independent of the choice of a fundamental solution of the second order system, and the
functions
\[
\operatorname{tr}(\mathcal{S}_{A,q}^r),\qquad 1\le r\le 5,
\]
define intrinsic invariants of the one-parameter family~\eqref{eq:cubic-family}.
In the explicit diagonal form~\eqref{eq:cubic-explicit-schwarzian} we have
\[
\operatorname{tr}(\mathcal{S}_{A,q}^r)(t)=2\,s(t)^r
=2\left(\frac{t(t^3+8)}{2(1-t^3)^2}\right)^r,
\qquad 1\le r\le 5.
\]  In
this way, periods of cubic threefolds give a higher-dimensional test case for the
non-abelian Schwarzian formalism, analogous to Dedekind's genus-one picture but now in
weight three and with genuinely non-abelian coefficient objects.

\part{The Mass-Spring System}
In this part we discuss mass--spring systems as a mechanical testing ground for the non-abelian Schwarzian, and use them to motivate the terminology ``classical'' and ``quantum'' connections.

\section{The classical mass--spring system}\label{sec:mass-spring-system}
The mass--spring system provides a concrete source of second-order equations with matrix coefficients.
It also serves as a simple model for the local-to-global viewpoint underlying our quantum formalism: the coefficients are best regarded as geometric objects on a \emph{time curve} rather than as functions of an absolute time parameter.

\subsection{What is time?}\label{subsec:why-quantum}

Imagine a universe in which time runs on an abstract Riemann surface $X$:
a universe without a beginning, without an end, without any distinguished moment,
and without any ``bangs.'' A viewpoint of this kind appears in no--boundary
proposal of Hartle--Hawking \cite{L35}. In that proposal, the question of what happened
before the Big~Bang becomes as meaningless as asking what lies south of the South~Pole.

A holomorphic time--chart on $X$ may be thought of as a local clock $t$,
while another chart provides a clock $\tau$. The two clocks are related by a
biholomorphic ``time difference'' $\lambda$ satisfying $t=\lambda(\tau)$.

\subsection{The Hooke's law}\label{subsec:mass-spring-libretexts}
We begin with a frictionless unit mass attached to a spring of stiffness $k>0$.
The restoring force is given by Hooke's law:
\begin{equation}\label{eq:hooke-law}
F=-k\,\psi.
\end{equation}
For a unit mass, Newton's equation therefore yields
\[
\ddot{\psi}+k\psi=0.
\]
To match the normalization used throughout the paper, we rewrite this as
\begin{equation}\label{eq:mass-spring-single}
\ddot{\psi}=q\,\psi,
\qquad
q=-k.
\end{equation}

\begin{figure}[H]
\centering
\begin{tikzpicture}[scale=1, thick]
 % wall
 \fill[gray!35] (-0.3,-0.6) rectangle (0,0.6);
 \draw (0,-0.6) -- (0,0.6);

 % spring 1
 \draw (0,0) -- (0.4,0);
 \draw[decorate, decoration={zigzag,segment length=4,amplitude=2}] (0.4,0) -- (2.4,0);
 \node[above] at (1.4,0.1) {$k_1$};

 % mass 1
 \draw[fill=gray!15] (2.4,-0.25) rectangle (3.2,0.25);
 \node at (2.8,0) {};

 % spring 2
 \draw (3.2,0) -- (3.6,0);
 \draw[decorate, decoration={zigzag,segment length=4,amplitude=2}] (3.6,0) -- (5.6,0);
 \node[above] at (4.6,0.1) {$k_2$};

 % mass 2
 \draw[fill=gray!15] (5.6,-0.25) rectangle (6.4,0.25);
 \node at (6.0,0) {};

 % displacement arrows
 \draw[->] (2.8,-0.7) -- (3.4,-0.7) node[midway,below] {$\psi_1$};
 \draw[->] (6.0,-0.7) -- (6.6,-0.7) node[midway,below] {$\psi_2$};
\end{tikzpicture}
\caption{Mass--spring system with $n=2$.}
\label{fig:two-mass-two-spring}
\end{figure}
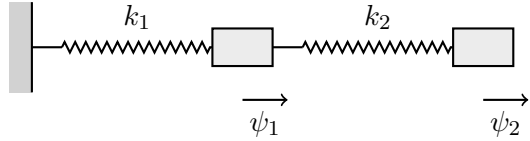

For a system with two degrees of freedom ($n=2$), with spring constants $k_1$ and $k_2$ , one obtains the coupled system
\begin{equation}\label{eq:libretexts-6211}
\begin{aligned}
\ddot{\psi}_{1} &= -(k_{1}+k_{2})\,\psi_{1} + k_{2}\,\psi_{2},\\[4pt]
\ddot{\psi}_{2} &= k_{2}\,\psi_{1} - k_{2}\,\psi_{2}.
\end{aligned}
\end{equation}
Equivalently, writing $\boldsymbol{\psi}=\begin{pmatrix}\psi_1\\ \psi_2\end{pmatrix}$ this is a $2\times 2$ second-order matrix ODE
\begin{equation}\label{eq:two-mass-matrix-second-order}
\ddot{\boldsymbol{\psi}} = q\,\boldsymbol{\psi},
\qquad
q=\begin{pmatrix} -k_1-k_2 & k_2\\ k_2 & -k_2\end{pmatrix}.
\end{equation}

In the presence of friction (\emph{damping}) the system takes the form
\begin{equation}\label{eq:damped-matrix-ode}
\ddot{\boldsymbol{\psi}}
= p\,\dot{\boldsymbol{\psi}} + q\,\boldsymbol{\psi},
\end{equation}
where $p$ is the damping matrix and $q$ is the stiffness matrix.

Now imagine that $p$ and $q$ vary meromorphically with respect to time. In a local clock $t$ we then have
\begin{equation}\label{eq:complex-time-damped}
\ddot{\boldsymbol{\psi}}=p(t)\,\dot{\boldsymbol{\psi}}+q(t)\,\boldsymbol{\psi},
\qquad p(t),q(t)\in M_n(\Cc).
\end{equation}

\subsection{Eternal mass--spring system}
Call such a system \emph{eternal} if it is compatible with changes of local clock, i.e.\ if it depends only on the time curve $X$ and not on the chosen coordinate.
Concretely, under a change of clock $t=\lambda(\tau)$, the chain rule above shows that \eqref{eq:complex-time-damped} retains the same form provided the coefficients transform by
\[
\widetilde p(\tau)=\lambda'(\tau)\,\bigl(p\circ\lambda\bigr)(\tau)+\frac{\lambda''(\tau)}{\lambda'(\tau)}\,I,
\qquad
\widetilde q(\tau)=(\lambda'(\tau))^{2}\,\bigl(q\circ\lambda\bigr)(\tau),
\]
where $I$ denotes the identity matrix.
In particular, $p$ transforms affinely (with an inhomogeneous $\lambda''/\lambda'$ term), while $q$ transforms tensorially with weight~$2$.
In our notation, for an eternal system the quantity $A=p/2$ transforms from one chart to another as a time-connection with eccentricity $e=\tfrac12$, while $q$ transforms as a time-quadratic differential. A natural coordinate-free invariant is obtained from the Schwarzian curvature. For $e=\tfrac12$ the curvature of $A$ is
\[
F_A = A' - A^{2},
\]
and the associated Schwarzian is $S_{A,q}:=F_A-q$. While $S_{A,q}$ has a scalar projective anomaly, its traceless part does not. Consequently the \emph{projective curvature}
\begin{equation}\label{eq:proj-curvature}
\mathcal{R}
:= S_{A,q} -\frac{1}{n}\tr(S_{A,q})\,I
\end{equation}
is a well-defined, eternal $\End(\Cc^n)$-valued meromorphic quadratic differential on $X$, independent of the local clock. Notice that $\mathcal{R}$ always vanishes when $n=1$.

\subsection{Quantum mass--spring systems}
In the quantum mass--spring system one replaces the matrix-valued coefficients $p$ and $q$ by local operators acting on the vector space of local holomorphic functions.
Under changes of local clock these operators are required to satisfy the quantum transformation laws of Definitions~\ref{def:quantum-connection} and~\ref{def:quantum-m-differential}, so that the damping term defines a quantum time-connection and the stiffness term defines a quantum time-quadratic differential.
From this viewpoint, a specific moment plays no distinguished role; what matters is how the local operators glue together globally on the time curve~$X$.

This resonates with a point-free theme in Grothendieck's writings: the structure of a space is encoded by the sheaf-theoretic data that glue on overlaps, and ``points'' are secondary, recovered only when needed from the function theory \cite{L34}.
Quantum theory exhibits an analogous shift of emphasis, replacing a classical phase point by an algebra of observables.
Our operator-valued coefficients, together with their cocycle laws under changes of clock, fit naturally into this algebraic viewpoint and help explain the terminology ``quantum'' in this paper.

\bigskip
\begin{flushright}
\textbf{Mehrzad Ajoodanian}\\
\texttt{mehrzad77@gmail.com}
\end{flushright}

\end{document}